\documentclass[a4paper,12pt,reqno,openany]{amsart} %
\usepackage[verbose]{geometry}
\pdfoutput=1                    %

\usepackage[english]{babel}
\usepackage[utf8]{inputenc}

\usepackage{mathtools}
\usepackage{amsmath, amssymb, amsthm}
\usepackage{enumerate}
\usepackage{esint}              %

\usepackage{stmaryrd}               %
\usepackage{mathrsfs}           %
\usepackage{cases}              %

\usepackage{graphicx}           %

\usepackage{tikz}
\usetikzlibrary{decorations, patterns, cd} %
\usepackage[hang, small, bf]{caption}

\usepackage[pdftex, colorlinks=true, linkcolor=blue, citecolor=magenta, urlcolor=blue]{hyperref}
\usepackage[noabbrev,capitalize]{cleveref} %
\crefformat{section}{\S#2#1#3} %
\crefformat{subsection}{\S#2#1#3}
\crefformat{subsubsection}{\S#2#1#3}

\renewcommand{\roman}[1]{%
  \textup{\uppercase\expandafter{\romannumeral#1}}%
}

\renewcommand{\restriction}{\mathord{\upharpoonright}}
\newcommand{\D}{ \, \mathrm{d}}

\DeclareMathAlphabet\EuScript{U}{eus}{m}{n}
\providecommand{\Real}{\ensuremath \Re\mathfrak{e}} %

\DeclareMathAlphabet\mathcalboondox{U}{BOONDOX-calo}{m}{n} %

\providecommand{\Tiso}{\ensuremath \mathsf{P}}

\providecommand{\transpose}{\ensuremath \intercal} %

\newcommand\SmallMatrix[1]{{%
    \tiny\arraycolsep=0.3\arraycolsep\ensuremath{\begin{bmatrix}#1\end{bmatrix}}}}

\usepackage{svg}
\usepackage{subcaption}         %

\hypersetup{
  pdftitle={The metric for matrix degenerate Kato square root operators},
  pdfauthor={Gianmarco Brocchi and Andreas Rosén},
  colorlinks=true,
  citecolor=[rgb]{0.682, 0.133, 0.188}, %
  linkcolor=[rgb]{0,0,0.5},             %
  urlcolor=[rgb]{0,0,0.75},     
  pdfstartview=FitH,
  pdfdisplaydoctitle=true,
  pdflang=en-GB,
  breaklinks=true
}

\usepackage[disable]{todonotes}
\usepackage{todonotes}
\usepackage[backend=biber,backref=true,style=alphabetic]{biblatex}
\AtEveryBibitem{%
  \clearlist{language}%
}
\usepackage{csquotes}           %
\theoremstyle{plain}
\newtheorem{theorem}{Theorem}[section]  %
\newtheorem{cor}[theorem]{Corollary}    %
\newtheorem{lemma}[theorem]{Lemma}
\newtheorem{prop}[theorem]{Proposition}

\theoremstyle{definition}
\newtheorem{define}[theorem]{Definition}
\theoremstyle{remark}
\newtheorem{remark}[theorem]{Remark}
\newtheorem{example}[theorem]{Example}

\numberwithin{equation}{section} %
\allowdisplaybreaks[4] %

\title{The metric for matrix degenerate Kato square root operators}
\author[Brocchi]{Gianmarco Brocchi}
\address{
  Háskóli Íslands \\ %
  Tæknigar{\dh}ur, Dunhagi 5 \\ %
  107 Reykjavík \\
  Iceland}
\email{gianmarco@hi.is}
\author[Rosén]{Andreas Rosén}
\address{
  Mathematical Sciences \\
  Chalmers University of Technology and the University of Gothenburg \\
  SE-412 96 Göteborg \\
  Sweden}
\email{andreas.rosen@chalmers.se}

\begin{document}
\begin{abstract}                %
  We prove a Kato square root estimate with anisotropically degenerate matrix coefficients.
  We do so by doing the harmonic analysis using an auxiliary Riemannian metric adapted to the operator.
  We also derive $L^2$-solvability estimates for boundary value problems
  for divergence form elliptic equations with matrix degenerate coefficients.
  Main tools are chain rules and Piola transformations
  for fields in matrix weighted $L^2$ spaces,
  under $W^{1,1}$ homeomorphism.
\end{abstract}

\maketitle
\tableofcontents%

\section*{Introduction}

Our point of departure is
the celebrated Kato square root estimate
\begin{equation}\label{eq:Kato_square_root_estimate}
  \lVert \sqrt{ - \mathrm{div} A \nabla  } u \rVert_{L^2(\mathbb{R}^d)} \eqsim \lVert \nabla u \rVert_{L^2(\mathbb{R}^d)}
\end{equation}
proved by \cite{SolKato2002},
where the complex-valued coefficient matrix $A$ is
assumed only
to be bounded, measurable and accretive.
After its formulation by Tosio Kato in \cite{zbMATH03183358}, %
\cite[p. 332]{KatoBook}, %
already the one-dimensional result, $d = 1$,
was only solved 20 years later by Coifman, McIntosh, and Meyer \cite{CMcM82}.
The higher dimensional result \cite{SolKato2002} in $d \ge 2$
took an additional 20 years,
and a reason was that
the non-surjectivity of $\nabla$ requires a more elaborated stopping time argument
in the Carleson measure estimate at the heart of the proof.
That the estimate \eqref{eq:Kato_square_root_estimate} is beyond the scope of classical Calderón--Zygmund theory for $d \ge 2$
is clear from the fact that, in general, the Kato square root estimate may hold in $L^p(\mathbb{R}^d)$
only for $p$ in a small interval around $p=2$,
depending on the matrix $A$.    %
See \cite[page 7]{auscher2007necessary}.

In this paper, we consider the extension of \eqref{eq:Kato_square_root_estimate}
to weighted $L^2$ estimates. Cruz-Uribe and Rios \cite{zbMATH06429144}
proved the weighted Kato square root estimate
\begin{equation}\label{eq:weighted_Kato}
  \lVert \sqrt{ - (1/w) \mathrm{div} A \nabla  } u \rVert_{L^2(\mathbb{R}^d, w)} \eqsim \lVert \nabla u \rVert_{L^2(\mathbb{R}^d, w)}
\end{equation}
for Muckenhoupt weight $w \in A_2(\mathbb{R}^d)$
and degenerate coefficient matrices $A$ satisfying
\begin{equation*}
   \Real \langle A(x) v,v \rangle \gtrsim w(x) \lvert v \rvert^2  \, , \quad \lvert A(x) \rvert \lesssim w(x) \quad \text{for all } x \in \mathbb{R}^d , v \in \mathbb{C}^d.
\end{equation*}
It should be noted that Rubio de Francia %
extrapolation is not applicable here, since the operator $- (1/w) \mathrm{div} A \nabla$
and the $L^2(w)$-norm are coupled.
However, under additional assumption on $w$,
Cruz-Uribe, Martell, and Rios \cite{CMR2018}
proved \eqref{eq:weighted_Kato} with degenerate coefficients also
in the unweighted $L^2(\mathbb{R}^d)$-norm.

We shall however follow a different path,
where we seek to decouple $A$ from $w$ in the operator  $- (1/w) \mathrm{div} A \nabla$.
To this end, we consider more general \emph{anisotropically} degenerate elliptic operators $- (1/a) \mathrm{div} A \nabla$,
where the complex-valued scalar function $a(x)$ is controlled by a scalar weight $\mu$, as
\begin{equation}\label{eq:condition_on_a}\tag{a}
  \Real \, a(x) \gtrsim \mu(x)  \, , \quad \lvert a(x) \rvert \lesssim \mu(x) 
\end{equation}
and the complex matrix function $A(x)$ is controlled as
\begin{equation}\label{eq:condition_on_A}\tag{A}
  \Real \langle A(x) v, v\rangle \gtrsim \langle W(x) v, v\rangle \, , \quad \lvert W(x)^{-1/2} A(x) W(x)^{-1/2} \rvert \lesssim 1
\end{equation}
by a matrix weight $W$, meaning that $W(x)$ is a positive definite matrix at almost every point $x \in \mathbb{R}^d$.
The second condition in \eqref{eq:condition_on_A} is equivalent to
\begin{equation*}
  \langle A(x) v, v\rangle \lesssim \langle W(x) v, v\rangle \quad \text{for all } x \in \mathbb{R}^d , v \in \mathbb{C}^d.
\end{equation*}
Note carefully that for such degenerate elliptic operators  $- (1/a) \mathrm{div} A \nabla$,
not only the size of the two coefficients $a$ and $A$ can differ unboundedly, %
but the size of $A(x)v$ can vary unboundedly between different directions $v \in \mathbb{C}^d$, $\lvert v \rvert = 1$,
at $x \in \mathbb{R}^d$.
\cref{fig:ellipses1} and \cref{fig:ellipses2} below show ellipses centred at a point $x$
whose principal axes are the eigenvectors of the matrix $A(x)$.
These are two examples of such anisotropically degenerate matrices $A(x)$,
which are discussed in more details in \cref{ex:ellipses} and \cref{ex:example2}.
\begin{figure}[th]
  \includegraphics[width=8cm]{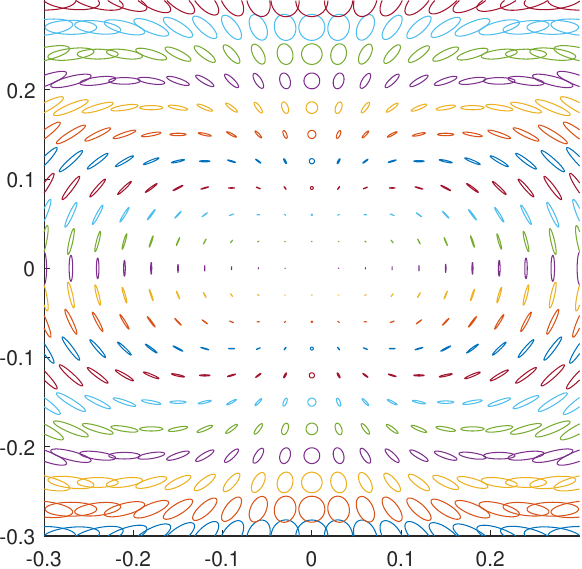}
  \caption{Geodesic disks in the metric of \cref{ex:ellipses} are ellipses whose principal axes are the
    eigenvectors of the matrix $A(x)$. These ellipses shrink anisotropically towards the origin.}
  \label{fig:ellipses1}
\end{figure}
\begin{figure}[th]
  \includegraphics[width=8cm]{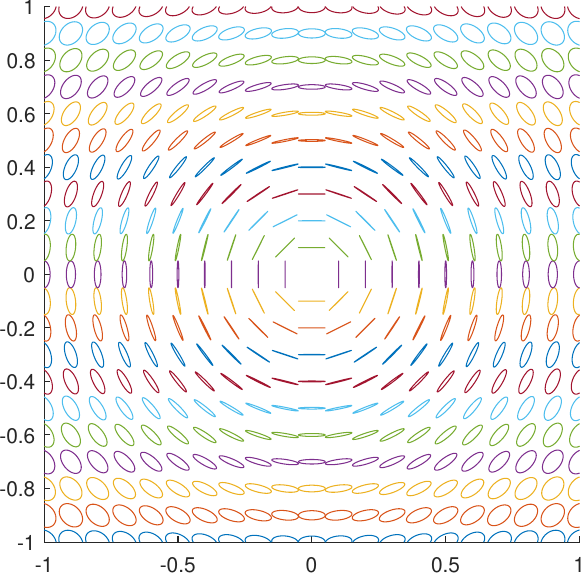}
  \caption{Geodesic disks in the metric of \cref{ex:example2} for $a = 1$ are ellipses with increasing eccentricity.}
  \label{fig:ellipses2}
\end{figure}

The natural norms for the operator $- (1/a) \mathrm{div} A \nabla$ appear
using the standard duality proof of the Kato square root estimate
in the special case of self-adjoint coefficients $a = \mu$ and $A = W$:
\begin{align*}
  \lVert \sqrt{ - (1/\mu) \mathrm{div} W \nabla  } & u \rVert_{L^2(\mu)}^2 \\
  & = \langle -(1/\mu) \mathrm{div} W \nabla u, u \rangle_{L^2(\mu)}
  = \langle W \nabla u, \nabla u \rangle_{L^2(\mathbb{R}^d)} \eqqcolon \lVert \nabla u \rVert_{L^2(\mathbb{R}^d, W)}^2.
\end{align*}
Note that the matrix-weighted space $L^2(\mathbb{R}^d, W)$ does not see the scalar weight $\mu$.
Our problem is thus to understand under what conditions on $\mu$ and $W$ 
the matrix-weighted Kato square root estimate
\begin{equation}\label{eq:matrix_weighted_Kato}
  \lVert \sqrt{ - (1/a) \mathrm{div} A \nabla  } u \rVert_{L^2(\mathbb{R}^d, \mu)} \eqsim \lVert \nabla u \rVert_{L^2(\mathbb{R}^d, W)}
\end{equation}
holds for general $a$ and $A$ satisfying \eqref{eq:condition_on_a} and \eqref{eq:condition_on_A} respectively.
We study \eqref{eq:matrix_weighted_Kato} using a framework of first-order differential operators,
which goes back to \cite{zbMATH01180060} and \cite{AKM^c}. %
The approach consists in proving boundedness of the $H^\infty$ functional calculus
for perturbations of a first-order self-adjoint differential operator $D$,
perturbed by a bounded and accretive multiplication operator $B$.
In our context, we set
\begin{equation}\label{eq:first_appearance_of_D}
  D =
  \begin{bmatrix}
    0 & - (1/\mu) \mathrm{div} W \\
    \nabla & 0 
  \end{bmatrix}\, , \quad
  B =
  \begin{bmatrix}
    \mu/a & 0 \\
    0 & W^{-1} A
  \end{bmatrix}.
\end{equation}
The operators $D$ and $B$ act on the Hilbert space $\mathcalboondox{H} = L^2(\mu) \oplus L^2(W)$.
The perturbed operator
\begin{equation}\label{eq:perturbed_operator}
  B D =
  \begin{bmatrix}
    0 & - (1/a) \mathrm{div} W \\
    W^{-1} A \nabla & 0 
  \end{bmatrix}
\end{equation}
has spectrum in a bisector around the real line,
and we show the boundedness of the $H^\infty$ functional calculus for $B D$,
as defined in \cref{subsec:HinftyFC}.
The Kato square root estimate \eqref{eq:matrix_weighted_Kato} then follows from the boundedness of the sign function of $B D$,
namely from the estimate
\begin{equation}
  \label{eq:fun_calc_estimate}
  \left\lVert \sqrt{(B D)^2}
  \SmallMatrix{u \\ 0}  
  \right\rVert_{\mathcalboondox{H}} \eqsim \Big\lVert B D
  \SmallMatrix{u \\ 0}
  \Big\rVert_{\mathcalboondox{H}}
\end{equation}
since $\sqrt{(B D)^2} = \mathrm{sgn}(B D) B D$
and
\begin{equation*}
  \sqrt{(B D)^2} =
  \begin{bmatrix}\textstyle
    \sqrt{-\frac{1}{a}\mathrm{div}A\nabla} & 0 \\
    0 & \sqrt{-W^{-1} A \nabla \frac{1}{a}\mathrm{div}W}
  \end{bmatrix},
\end{equation*}
while the right hand side of \eqref{eq:fun_calc_estimate} is equivalent to $\lVert \nabla u \rVert_{L^2(W)}$ as desired.

The proof of \eqref{eq:Kato_square_root_estimate} from \cite{SolKato2002} uses a local $T b$ theorem for square functions,
with test functions $b$ constructed using the elliptic operator, which reduces the problem to a Carleson measure estimate.
In the isotropically degenerate case with $W = \mu I$,
boundedness of the $H^\infty$ functional calculus of $B D$,
and in particular \eqref{eq:fun_calc_estimate},
was proved in \cite{ARR}.
Important to note is that the proof in \cite{ARR}
does not require $B$ to be block diagonal,
as compared to the one in \cite{CUR2015}, %
as \cite{ARR} uses a more elaborate double stopping argument
for test function and weight.
Our results in the present paper do not require $B$ to be block diagonal either.
Non-block diagonal $B$ are important in applications to boundary value problems:
see \cite{AAMc2010}, \cite{AA2011} and references therein.
We extend \cite[\S 4]{AMR} to anisotropic degenerate elliptic equations in \cref{sec:BVPs}.

When trying to prove boundedness of the $H^\infty$ functional calculus for our operator $B D$ from \eqref{eq:perturbed_operator},
following the local $Tb$ argument in \cite{ARR},
one soon realises that
the main obstacle when $W \neq \mu I$
is the $L^2$ off-diagonal estimates for the resolvents of $B D$.
In all previous works, one has an estimate %
\begin{equation}
  \label{eq:off-diagonal_for_resolvent}
  \lVert (I + i t B D)^{-1} u \rVert_{L^2(F)} \lesssim \eta\Big( \frac{\mathrm{dist}(E,F)}{t} \Big) \, \lVert u \rVert_{L^2(E)}
\end{equation}
with $\eta(x)$ rapidly decaying to $0$ as $x \to \infty$
and $\mathrm{dist}(E,F)$ being the distance between sets $E,F \subseteq \mathbb{R}^d$.
So the resolvents are not only bounded,
but act almost locally at scale $t$. %
When $W \neq \mu I$, this crucial estimate in the local $Tb$ theorem may fail.
Indeed, the commutator estimate used in the proof of \eqref{eq:off-diagonal_for_resolvent} fails,
as it requires the boundedness of
\begin{equation*}
  [D, \eta] =
  \begin{bmatrix}
    0 & - \frac{1}{\mu}[\mathrm{div},\eta] W \\
    [\nabla,\eta] & 0 
  \end{bmatrix}.
\end{equation*}
This is a bounded multiplier on $L^2(\mu) \oplus L^2(W)$, with norm $\lVert \nabla \eta \rVert_{L^\infty}$, only if $\lvert W \rvert \lesssim \mu$.
But even assuming this latter bound, %
it is still unclear to us how to extend the remaining part of the Euclidean proof from \cite{ARR}
which seems to require non-trivial two-weight bounds.

The way we instead resolve this problem
is to replace the Euclidean metric with a Riemannian metric $g$ adapted to the operator $B D$.
We show in \cref{sec:higher_dim}
that the Euclidean operator $B D$ on $L^2(\mathbb{R}^d,\mu) \oplus L^2(\mathbb{R}^d;\mathbb{C}^d, W)$ is in fact similar
to an operator $B_M D_M$ acting on $L^2(M,\nu) \oplus L^2(TM,\nu I)$
for a auxiliary Riemannian manifold $M$ with metric $g$
and a single scalar weight $\nu$ associated with $\mu,W$. %
\begin{figure}[h]\label{fig:diagram-similarity}
  \begin{tikzcd}
    \mathcalboondox{H}_M \coloneqq L^2(M,\nu) \oplus L^2(TM,\nu I) \arrow[d, "\Tiso"] \arrow[r, "D_M B_M "] & \mathcalboondox{H}_M \\
    \mathcalboondox{H} \coloneqq L^2(\mathbb{R}^d,\mu) \oplus L^2(\mathbb{R}^d;\mathbb{C}^d, W) \arrow[r, "D B "] & \mathcalboondox{H} \arrow[u, "\Tiso^{-1}"]
  \end{tikzcd}
  \caption{We will use a unitary map $\Tiso$ and its inverse,
    introduced in \cref{sec:1dim} and defined in \eqref{eq:definition_rubber_band_higher_dim}.}
\end{figure}

Note that the scalar weight $\nu$ determines the norms \emph{both} on scalar \emph{and} vector functions.
Thus we have reduced to the situation in \cite{ARR}, %
but with $\mathbb{R}^d$ replaced by a manifold $M$.
The Euclidean proof in \cite{ARR} has been generalised to a class of manifolds in \cite{AMR},
notably those with positive injectivity radius and Ricci curvature bounded from below.
Applying \cite{AMR} to $B_M D_M$ gives boundedness of its $H^\infty$ functional calculus and,
via similarity, also for 
our anisotropically degenerate operator $B D$ on $\mathbb{R}^d$.
This in particular shows the matrix-weighted Kato square root estimate \eqref{eq:matrix_weighted_Kato}
for a class of weights $(\mu,W)$ determined by properties of $(g,\nu)$.
The examples at the end of \cref{sec:higher_dim} 
show that indeed this class covers weights beyond \cite{ARR}.
In a forthcoming paper,
we shall relax further the hypotheses on the auxiliary manifold $(M,g)$.
\section*{Preliminaries}%

\subsection*{Notations}
For two quantities $X,Y \ge 0$,
the expression $X \lesssim Y$ means that there exists a finite, positive constant $C$ such that $X \le C Y$.
The expression $X \gtrsim Y$ means $Y \lesssim X$.
When both expressions hold simultaneously, with possibly different constants,
we will write $X \eqsim Y$.
Given a matrix $W$ %
the quantities $\lvert W \rvert$ and $\lVert W \rVert_{\mathrm{op}}$ denote any of the equivalent matrix norms of $W$.

As discussed in the introduction,
the Kato square root estimate follows from
the boundedness of functional calculus for a bisectorial operator $B D$.
Here we recall these concepts.

\subsection{Bisectorial operators}
For an angle $\theta \in [0,\pi/2)$, consider the closed bisector
\begin{equation*}
  S_{\theta} \coloneqq \{ z \in \mathbb{C} \, :\, \lvert \mathrm{arg}(z) \rvert \le \theta \} \cup \{ 0 \} \cup \{ z \in \mathbb{C} \, :\, \lvert \mathrm{arg}(- z) \rvert \le \theta \} .
\end{equation*}

\begin{define}[Bisectorial operator]
  A closed, densely defined operator $D$ on a Hilbert space
  is bisectorial if there exists an angle $\theta \in [0,\pi/2)$ such that
  \begin{itemize}
  \item the spectrum $\sigma(D)$ is contained in the bisector $S_{\theta}$;
  \item outside $S_{\theta}$ we have resolvent bounds: $\lVert (\lambda I - D)^{-1} \rVert \lesssim 1/\mathrm{dist}(\lambda, S_\theta)$.
  \end{itemize}
\end{define}
Given a densely defined operator $D$, its domain will be denoted by $\mathsf{dom}(D)$. 
If $D$ is bisectorial, we have the topological (not necessarily orthogonal) splitting \cite[Proposition 3.3 (ii)]{AAMc2010}
\begin{equation*}              
  \mathcalboondox{H} = \mathsf{ker}(D) \oplus \overline{\mathsf{im}(D)}
\end{equation*}
where $\mathsf{ker}(D) \coloneqq \{ u \in \mathsf{dom}(D), Du = 0 \}$ is always closed and $\mathsf{im}(D) \coloneqq \{ Du \in \mathcalboondox{H}, u \in \mathsf{dom}(D) \}$. 
In particular, restricting $D$ to the closure of its range gives an injective bisectorial operator.

\subsection{Bounded holomorphic functional calculus}\label{subsec:HinftyFC}
Given $\theta' > \theta$, with $\theta',\theta \in [0,\pi/2)$,
let $\mathring{S_{\theta'}}$ be the interior of the bisector $S_{\theta'}$.
Denote by $H^\infty(\mathring{S_{\theta'}})$ the space of bounded holomorphic functions on $\mathring{S_{\theta'}}$.
Given an injective operator $D$ which is bisectorial on $S_{\theta}$, 
we say that $D$ has bounded $H^\infty$ functional calculus on $\mathring{S_{\theta'}}$ if for all function $f \in H^\infty(\mathring{S_{\theta'}})$
we can define a bounded operator $f(D)$ with norm bound
\begin{equation*}
  \lVert f(D) \rVert_{\mathcalboondox{H} \to \mathcalboondox{H}} \lesssim %
  \lVert f \rVert_{L^\infty(\mathring{S_{\theta'}})}.
\end{equation*}

For a non-injective operator $D$,
the $H^\infty$ functional calculus can be extended to the whole space $\mathcalboondox{H}$
by setting $f(D)\restriction_{\mathsf{ker}(D)} \coloneqq f(0) I\restriction_{\mathsf{ker}(D)}$,
for $f \colon \{0\} \cup \mathring{S_{\theta'}} \to \mathbb{C}$ such that $f\restriction_{\mathring{S_{\theta'}}} \in H^{\infty}(\mathring{S_{\theta'}})$.

\subsection{Quadratic estimates}%
Let $\psi$ be any function in $H^\infty(\mathring{S_{\theta'}})$
which is non-vanishing on both sectors and
decaying as $\lvert \psi(\zeta) \rvert \lesssim \lvert \zeta \rvert^s (1+\lvert \zeta \rvert^{2s})^{-1}$ for some $s>0$.
We call the class of such functions $\Psi(\mathring{S_{\theta'}})$.
A bisectorial operator $D$ acting on a Hilbert space $\mathcalboondox{H}$ satisfies quadratic estimates  if
\begin{equation}\label{eq:QE_prelim}
  \Big( \int_0^\infty \lVert \psi_t(D) u \rVert_{\mathcalboondox{H}}^2 \frac{\D{t}}{t} \Big)^{1/2} \lesssim \lVert u \rVert_{\mathcalboondox{H}} 
\end{equation}
holds for all $u \in \mathcalboondox{H}$ and all $\psi \in \Psi(\mathring{S_{\theta'}})$,
where $\psi_t(\zeta) \coloneqq \psi(t \zeta)$.
If $D$ satisfies \eqref{eq:QE_prelim} for one such $\psi$, then \eqref{eq:QE_prelim} holds for all $\psi \in \Psi(\mathring{S_{\theta'}})$.
Then, for simplicity, we will take $\psi(\zeta) = \zeta / (1+ \zeta^{2})$.
Bisectorial operators $D$, for which both $D$ and $D^\star$ satisfy the quadratic estimates \eqref{eq:QE_prelim}
have a bounded $H^\infty$ functional calculus. See \cite[\S 3.(E)]{zbMATH00923920}, where this is shown for sectorial operators.
The extension to bisectorial operators is straightforward. See also \cite[\S 6.1]{AA2011} for a short derivation of the needed estimates.

\subsection{Weights}
A scalar weight is a function $x \mapsto \mu(x)$ which is positive almost everywhere,
while a matrix weight is a matrix-valued function $x \mapsto W(x)$ such that
$W(x)$ is symmetric, positive definite matrix at almost every $x$.
We will consider weights on $\mathbb{R}^d$ and, more generally, on a complete Riemannian manifold $M$ with Riemannian measure $\mathrm{d}y$.

\begin{define}\label{def:W-boundedness_and_accretivity}
  Let $W$ be a matrix weight.
  A multiplication operator $B$ is said to be $W$-bounded if
  \begin{equation*}
    \lvert W^{1/2} B W^{-1/2} \rvert \lesssim 1 \quad \text{ a.e. }
  \end{equation*}
  and it is said to be $W$-accretive if
  \begin{equation*}
    \Real \langle W^{1/2} B W^{-1/2} v , v \rangle \gtrsim \lvert v \rvert^2 \quad \text{ a.e. and } \forall v \in \mathbb{C}^d.
  \end{equation*}
\end{define}
Note that
\begin{itemize}
\item $B$ is $W$-bounded if and only if
  the map $v \mapsto B v$ is bounded in the norm $v \mapsto \lvert W^{1/2} v\rvert$
\item $B$ is $W$-accretive if and only if
  the map $v \mapsto B v$ is accretive with respect to the inner product $\langle W v , v\rangle$ associated to the norm $\lvert W^{1/2} v\rvert$.
\item For scalar weights $W=w$ this reduces to standard unweighted notions of boundedness and accretivity.
\end{itemize}
When $W$ is a block diagonal matrix $\begin{bsmallmatrix} 
\mu & 0 \\
0 & w
\end{bsmallmatrix}$,
we will use the notation $(\mu \oplus w)$,
and say that a multiplication operator is $(\mu \oplus w)$-bounded and $(\mu \oplus w)$-accretive.

A special class of weights are Muckenhoupt weights, which are defined in terms of averages.
Let $B = B(x,r)$ be a geodesic ball of radius $r>0$ centred at $x$.
If $\lvert B \rvert$ denotes the Riemannian measure of a ball $B$,
the average of a scalar weight $\nu$ over $B$ is $\fint_B \nu \D{y} \coloneqq \lvert B \rvert^{-1} \int_B \nu \D{y}$.
\begin{define}[Muckenhoupt $A_2^R$ weights]\label{def:A_2^R}                  %
  Let $R>0$ be fixed.
  A scalar weight $\nu \colon M \to [0,\infty]$ belongs to the Muckenhoupt class $A_2^R(M)$,
  with respect to the Riemannian measure $\mathrm{d}y$, if
  \begin{equation*}
    [\nu]_{A_2^R} \coloneqq \sup_{\substack{y_0 \in M \\ r < R}} \Big(\fint_{B(y_0,r)} \nu(y) \D{y} \Big) \Big(\fint_{B(y_0,r)} \frac1{\nu(y)} \D{y} \Big) < \infty .
  \end{equation*}
  We say that a weight $\nu \in A_2(M)$ if
  $[\nu]_{A_2} \coloneqq \sup_{R > 0} [\nu]_{A_2^R} $ is finite.
\end{define}

We also introduce local Muckenhoupt weights,
as these are used to apply dominated convergence locally,
for example in proving the density of smooth functions in matrix-weighted Sobolev spaces.
Note that we do not use the $A_2^{\mathrm{loc}}$ property quantitatively.

\begin{define}[Local Muckenhoupt weights]\label{def:local_A_2}
  Let $\Omega \subseteq \mathbb{R}^d$ be an open set, and let $\mu$ and $W$ be a scalar and a matrix weight, respectively.
  We say that $\mu$ is in $A_2^{\mathrm{loc}}(\Omega)$ if for any compact $K \subset \Omega$
  \begin{equation*}
    \sup_{B \subset K} \Big(\fint_B \mu(x) \mathrm{d}x \Big) \Big(\fint_{B} \frac{1}{\mu(x)} \mathrm{d}x \Big) < \infty ,
  \end{equation*}
  where the supremum is over balls $B$. %
  Similarly, $W$ is in  $A_2^{\mathrm{loc}}(\Omega)$ if for any compact $K \subset \Omega$
  \begin{equation*}
    \sup_{B \subset K} \left\lVert \Big(\fint_B W(x) \mathrm{d}x \Big)^{1/2} \Big(\fint_{B} W^{-1}(x) \mathrm{d}x \Big)^{1/2} \right\rVert_{\mathrm{op}}^2 < \infty
  \end{equation*}
  where $\lVert \, \cdot\,\rVert_{\mathrm{op}}$ is the operator norm on the space of linear operators acting on $\mathbb{C}^d$.
\end{define}
As in \cref{def:local_A_2} we define $A_2^{\mathrm{loc}}(M)$ on a manifold $M$ for scalar weights.
One can show that for scalar weights it holds that $ A_2^R \subset A_2^{\mathrm{loc}}$ for any $R > 0$.

Defining matrix weights on a Riemannian manifold $M$ is more subtle.
At any $y \in M$, $W(y)$ should be a positive definite map of $T_y M$,
and in a chart $\varphi \colon \mathbb{R}^d \to M$,
it should be represented by $W_{\varphi} \coloneqq (\mathrm{d}\varphi)^{-1} W (\mathrm{d}\varphi^\star)^{-1}$.
However, the following example indicates that
the matrix $A_2$ condition on $W_{\varphi}$
is not in general invariant under transition maps between different smooth charts $\varphi$.
\begin{example}
  Let $W \colon \mathbb{R} \to \mathbb{R}^{2\times 2}$ be the matrix weight
  \begin{equation*}
    W(x) = \begin{bmatrix}
      \cos(x) & \sin(x) \\
      - \sin(x) & \cos(x)
    \end{bmatrix}
    \begin{bmatrix}
      1 & 0 \\
      0 & 1 + 2r
    \end{bmatrix}
    \begin{bmatrix}
      \cos(x) & -\sin(x) \\
      \sin(x) & \cos(x)
    \end{bmatrix}. 
  \end{equation*}
  The constant diagonal matrix $W(0) = 
  \begin{bsmallmatrix}      
    1 & 0 \\
    0 & 1 + 2r 
  \end{bsmallmatrix}$
  is trivially a matrix $A_2$ weight with $[W(0)]_{A_2} = 1$ for any $r \geq 0$.
  A direct computation shows that
  \begin{equation*}
        \lim_{r \to + \infty} \left\lVert \Big(\fint_0^{\pi} W(x) \mathrm{d}x \Big)^{1/2} \Big(\fint_0^{\pi} W^{-1}(x) \mathrm{d}x \Big)^{1/2} \right\rVert_{\mathrm{op}}^2 = \infty.
  \end{equation*}
  See also \cite[Proposition 5.3]{zbMATH01679403} %
  and \cite[Example 4.3]{zbMATH06728802}. %
\end{example}
Therefore, we make the following auxiliary definition:
\begin{define}
  A matrix weight $W \in \mathsf{End}(T M)$ belongs to $A_2^{\mathrm{loc}}(M)$ if at each $y \in M$
  there exists a chart $\varphi$ such that $(\mathrm{d}\varphi)^{-1} W (\mathrm{d}\varphi^\star)^{-1}$ %
  is a weight in $A_2^{\mathrm{loc}}(\mathbb{R}^d)$.
\end{define}

\section{Two scalar weights in one dimension}
\label{sec:1dim}

Following the historical tradition of the Kato square root problem,
we first consider the one-dimensional problem.
We treat this case separately since                  %
all one-dimensional manifolds are locally isometric, %
so no hypothesis on the Riemannian metric $g$ is needed,
only hypothesis on the weight $\nu$.

In dimension $d = 1$ the matrix weight $W(x)$
reduces to a scalar weight $w(x)$,
and $\nabla = \mathrm{div} = \partial_x$ is the derivative.
Consider the differential operator
\begin{equation}\label{eq:D_in_1D}
  D =
  \begin{bmatrix}
    0 & - (1/\mu) \partial_x w \\
    \partial_x & 0
  \end{bmatrix}.
\end{equation}
Let $\rho \colon \mathbb{R} \to \mathbb{R}$ be a ``rubber band'' parametrisation, a map stretching the real line, with $y = \rho(x)$ for $x \in \mathbb{R}$.
To see $g,\nu$ appear, we consider the pullback
\begin{equation}\Tiso \colon 
  \begin{bmatrix}
    v_1(y) \\
    v_2(y) 
  \end{bmatrix} \mapsto
  \begin{bmatrix}
    v_1(\rho(x)) \\
    v_2(\rho(x)) \rho'(x)
  \end{bmatrix}
  =
  \begin{bmatrix}
    u_1(x) \\
    u_2(x)
  \end{bmatrix}
  .\label{eq:rubber_band_pullback}
\end{equation}

The basic observation is the following.
\begin{lemma}\label{lemma:rubber_band}
  Let $\mu,w$ be two weights that are smooth on an interval $I \subset \mathbb{R}$.
  Let $\rho \colon I \to \mathbb{R}$ be such that $\rho'(x) = \sqrt{\mu(x)/w(x)}$.
  Set $M \coloneqq \rho(I) \subset \mathbb{R}$. Let $ \nu(\rho(x)) \coloneqq \sqrt{\mu(x) w(x)} $
  and
  \begin{equation}\label{eq:DM_in_1D}
    D_M \coloneqq
    \begin{bmatrix}
      0 & - (1/\nu)\partial_y \nu \\
      \partial_y & 0
    \end{bmatrix}.
  \end{equation}
  Then the map $\Tiso$ defined in \eqref{eq:rubber_band_pullback} is an isometry  between the Hilbert spaces
  $\mathcalboondox{H} = L^2(I,\mu) \oplus L^2(I,w)$ and $\mathcalboondox{H}_M \coloneqq L^2(M,\nu) \oplus L^2(M,\nu)$,
  and $\Tiso^{-1} D \Tiso = D_M$.

\end{lemma}
\begin{proof}
  We verify that $\Tiso D_M = D \Tiso$.
  This amounts to checking the equality $(!)$ in
  \begin{equation*}
    \Tiso D_M
    \begin{bmatrix}
      v_1 \\ v_2
    \end{bmatrix}
    =
    \begin{bmatrix}
      \Big( (-1/\nu) \partial_y \nu v_2 \Big) \circ \rho \\
      \rho' (\partial_y v_1) \circ \rho 
    \end{bmatrix}
    \overset{(!)}{=}
    \begin{bmatrix}
      - (1/\mu) \partial_x w (v_2 \circ \rho) \, \rho' \\
      \partial_x (v_1 \circ \rho) 
    \end{bmatrix}
    = D \Tiso 
    \begin{bmatrix}
      v_1 \\ v_2
    \end{bmatrix}.
  \end{equation*}
  The identity for the second component is the chain rule
  in \cref{prop:chain_rule_scalar} in one dimension.
  The identity for the first component is seen by multiplying and dividing by $\rho'$; 
  \begin{equation*}
    \frac{1}{\nu(\rho(x))\rho'(x)} \cdot \rho'(x) \partial_y( \nu \, v_2)(\rho(x)) \overset{(!)}{=} \frac{1}{\mu(x)} \partial_x( w(x) \rho'(x) (v_2 \circ \rho)(x) ) ,
  \end{equation*}
  and noting that %
  \begin{equation}\label{eq:1d_identities}
    \begin{cases}
      \mu(x) = \nu(\rho(x)) \rho'(x) , \\
      w(x) \rho'(x) = \nu(\rho(x)) .
    \end{cases}
  \end{equation}
  Using the identities in \eqref{eq:1d_identities} and the definition of $\Tiso$,
  the weighted norms $\lVert u_1 \rVert_{L^2(\mu)}$ and $\lVert u_2 \rVert_{L^2(w)}$ become
  \begin{align*}
    \int \lvert u_1(x) \rvert^2 \mu(x) \D{x} & = \int \lvert v_1(\rho(x)) \rvert^2 \nu(\rho(x)) \rho'(x) \D{x} = \int \lvert v_1(y) \rvert^2 \nu(y) \D{y} , \\
    \int \lvert u_2(x) \rvert^2 w(x) \D{x} & = \int \lvert v_2(\rho(x)) \rho'(x) \rvert^2 w(x) \D{x} = \int \lvert v_2(y) \rvert^2 \nu(y) \D{y} .
  \end{align*}
  This shows that $\Tiso$ is an isometry and concludes the proof.
\end{proof}

\cref{lemma:rubber_band} shows that formally,
$D$ in $L^2(I,\mu) \oplus L^2(I,w)$ is similar to $D_M$ defined in \eqref{eq:DM_in_1D} %
acting on $L^2(M,\nu) \oplus L^2(M,\nu)$,
to which \cite{ARR} applies.
To this end,
for non-smooth $\mu$ and $w$,
we need that $\nu \in A_2(\mathbb{R},\mathrm{d}y)$ and
the map $\rho$ to be absolutely continuous (in order to apply change of variables and chain rule as in \cref{appx:chainrule}), which amounts to $\rho' = \sqrt{\mu/w} \in L^1_{\mathrm{loc}}$.
This holds in particular if $\mu$,$w \in A_2^{\mathrm{loc}}$ which we need in order to apply \cref{prop:chain_rule_scalar}.
Note that, since $\mu$,$w^{-1}$ are in $L^1_{\mathrm{loc}}$, by Cauchy--Schwarz $\rho' \in L^1_{\mathrm{loc}}$ too.
Somewhat more subtle,
to ensure that we obtain a complete manifold $M$,
we must also take into account the completeness of the image of $\rho$ which, in the one-dimensional case, is the $y$-axis. See also \cref{ex:power_weights}.
This corresponds to the problem of defining $D$ as self-adjoint operator in $L^2(\mu) \oplus L^2(w)$.
Indeed, if $\rho$ maps onto an interval $M \subsetneq \mathbb{R} $,
boundary conditions
need to be imposed for $D_M$ to be self-adjoint in $\mathcalboondox{H}_M$,
and hence for $D = \Tiso D_M \Tiso^{-1}$ to be self-adjoint.
Although this can be done,
here we limit our study to the case in which $M$ is a complete manifold. %
See also \cref{ex:power_weights} below.

\begin{theorem}\label{thm:main_theorem1D}
  Consider a possibly unbounded interval $I = (c_1,c_2) \subseteq \mathbb{R}$.
  Let $\mu,w$ be  weights in $A_2^{\mathrm{loc}}(I)$ and assume that
  \begin{equation*}
    \int_{c_1}^{c} \sqrt{\frac{\mu}{w}} \D{t} = \int_{c}^{c_2} \sqrt{\frac{\mu}{w}} \D{t} = \infty \quad \text{ for } c_1 < c < c_2. 
  \end{equation*}
  For some fixed $c \in (c_1, c_2)$, let
  \begin{equation*}
    \rho(x) = \int_{c}^{x} \sqrt{\frac{\mu}{w}} \D{t} \quad \text{ and } \quad \nu(y) \coloneqq \sqrt{\mu(\rho^{-1}(y)) w(\rho^{-1}(y))}.
  \end{equation*}
  Assume that $\nu \in A_2(\mathbb{R}, \mathrm{d}y)$.
  Let $D$ be the operator defined in \eqref{eq:D_in_1D}
  and let $B$ be a $(\mu \oplus w)$-bounded
  and $(\mu \oplus w)$-accretive multiplication operator on $L^2(I,\mu) \oplus L^2(I,w)$ as in \cref{def:W-boundedness_and_accretivity}.
  Then $B D$ and $D B$ are bisectorial operators
  satisfying quadratic estimates and have bounded $H^\infty$ functional calculus in $L^2(I,\mu) \oplus L^2(I,w)$.
\end{theorem}
\begin{proof}
  The operator $D_M$ in \eqref{eq:DM_in_1D} has domain $\mathcal{H}^1_\nu \oplus (\mathcal{H}^1_\nu)^\star$ where
  \begin{equation*}
    \mathcal{H}^1_\nu \coloneqq \big\{ v \in L^2(\nu) \, : \, \partial_y v \in L^2(\nu) \big\}
  \end{equation*}
  and the adjoint space $(\mathcal{H}^1_\nu)^\star = \{ v \in L^2(\nu) \, : \, (1/\nu)\partial_y\nu v \in L^2(\nu) \}$.
  This space is isometric to the domain of $\partial_y$ in $L^2(1/\nu)$.
  See \cite[Lemma 2.3]{AMR} %
  which shows that the operators $\nabla$ and $\mathrm{div}$
  -- and in particular $\partial_y$ in one dimension -- have dense domains and are closed operators.
  The operator $D$ has domain $\mathcal{H}^1_{\mu,w} \oplus (\mathcal{H}^1_{\mu,w})^\star$ where
  \begin{equation*}
    \mathcal{H}^1_{\mu,w} \coloneqq \{ u \in L^2(\mu) \, : \, \partial_x u \in L^2(w) \}
  \end{equation*}
  and the adjoint space $(\mathcal{H}^1_{\mu,w})^\star = \{ u \in L^2(w) \, : \, (1/\mu)\partial_x w u \in L^2(\mu) \}$.
  Note that the operator $(1/\mu)\partial_x w \colon L^2(w) \to L^2(\mu)$ is unitary equivalent to
  $\partial_x \colon L^2(w^{-1}) \to L^2(\mu^{-1})$,
  since the multiplication by $w$ is a unitary map from $L^2(w) \to L^2(w^{-1})$.

  The pullback transformation $\Tiso$ maps between the domains of $D_M$ and $D$.
  Indeed, if $v \in \mathcal{H}^1_\nu$, then
  by \cref{prop:chain_rule_scalar} applied with $v = \nu$ and $V = \nu$,
  we have that $u \coloneqq \rho^* v \in L^2(\mu)$ and
  \begin{equation*}
    \partial_x u = \partial_x ( \rho^* v ) = \rho^* (\partial_y v) = \rho' \, (\partial_y v)\circ \rho \in L^2(w)
  \end{equation*}
  since $v_\rho = \mu$ and $V_\rho = w$.
  Similarly,
  we see that the $L^2$-adjoint of $\rho^*$, $\rho_*/\rho'$, maps
  \begin{equation*}
    \{ u \in L^2(w^{-1}) \, :\, \partial_x u \in L^2(\mu^{-1}) \} \to \mathcal{H}^1_{\nu^{-1}}.
  \end{equation*}
  By applying \cref{prop:chain_rule_vector} with $v = w^{-1}$ and $V = \mu^{-1}$,
  we see that both $v^{\rho}$ and $V^{\rho}$ equals $1/\nu$, so
  we have that $(\rho')^{-1} \rho_* u \in L^2(\nu^{-1})$ and
  \begin{equation*}
    \partial_y \big(\frac{\rho_*}{\rho'} u\big) = \frac{\rho_*}{\rho'} ( \partial_x u ) \in L^2(\nu^{-1}).
  \end{equation*}

  Let $B_M \coloneqq \Tiso^{-1} B \Tiso$.
  We show that $B$ is $(\mu \oplus w)$-bounded
  and $(\mu \oplus w)$-accretive if and only if the operator
  $B_M$ is $(\nu \oplus \nu)$-bounded
  and $(\nu \oplus \nu)$-accretive. %
  The $(\nu \oplus \nu)$-boundedness of $B_M$ means that
   \begin{equation}\label{eq:muw_bddness}
     \int \left(
       \begin{bsmallmatrix}
         \nu & 0 \\
         0 & \nu
       \end{bsmallmatrix}
       \Tiso^{-1} B \Tiso v , \Tiso^{-1} B \Tiso v
     \right) \D{y} \lesssim
     \int \big\lvert \begin{bsmallmatrix}
         \nu & 0 \\
         0 & \nu
       \end{bsmallmatrix}^{1/2} v
       \big\rvert^2 \D{y}.
   \end{equation}
   Let $u = \Tiso v$, then the left hand side of \eqref{eq:muw_bddness} equals
   \begin{equation*}
     \langle \Tiso^{-1} B u , \Tiso^{-1} B u \rangle_{L^2\big(
       \begin{bsmallmatrix}
         \nu & 0 \\
         0 & \nu
       \end{bsmallmatrix}\big)} =
     \langle \Tiso \Tiso^{-1} B u , B u \rangle_{L^2\big(
       \begin{bsmallmatrix}
         \mu & 0\\
         0  & w
       \end{bsmallmatrix}\big)} 
   \end{equation*}
   where we used that $\Tiso^{-1} = \Tiso^\star$,
   since $\Tiso$ is an isometry, as shown in \cref{lemma:rubber_band}.
   The same applies to show that
   $B_M$ is $(\nu \oplus \nu)$-accretive if and only if $B$ is $(\mu \oplus w)$-accretive.

   Now, to prove the theorem,
   apply  \cite[Theorem 3.3]{ARR} to $D_M B_M$, %
   where $D_M \coloneqq \Tiso^{-1} D \Tiso$.
   It follows that $D_M B_M$ satisfies quadratic estimates.
   The same holds for the operator $D B$ via the isometry $\Tiso$,
   and for $B D = B (D B) B^{-1}$.
 \end{proof}

\begin{remark}
  Since the Riemannian measure of $\rho(J)$ for any subinterval $J \subseteq I$ is $\int_J \rho'(x) \D{x}$,
  the condition $\nu \in A_2(\mathbb{R},\mathrm{d}y)$ explicitly means
  that for all intervals $J$, we have        %
  \begin{equation}\label{eq:nu_in_A2_means}
    \Big(\int_J \mu(x) \D{x}\Big)\Big( \int_J \frac{1}{w(x)} \D{x} \Big) \lesssim \Big( \int_J \sqrt{\frac{\mu}{w}} \D{x} \Big)^2.
  \end{equation}
  Note that the hypothesis
  $\mu$,$\mu^{-1}$,$w$,$w^{-1} \in L^1_{\mathrm{loc}}$
  and more precisely $\mu,w \in A_2^{\mathrm{loc}}$,
  is not used quantitatively,
  but only to ensure that:
  \begin{enumerate}
  \item $L^2(I,\mu)$ and $L^2(M,\nu)$ are contained in $L^1_{\mathrm{loc}}(\mathrm{d}x)$,
    so that the derivatives in the operator $D$ can be, and are, interpreted in the sense of distributions;
  \item the isometry $\Tiso$ maps $\mathsf{dom}(D_M)$ bijectively onto $\mathsf{dom}(D)$.
  \end{enumerate}
  A way to extend \cref{thm:main_theorem1D} to more rough weights
  would be to define the domain $\mathsf{dom}(D)$ as the image of $\mathsf{dom}(D_M)$ under the isometry $\Tiso$.
  In this way, one only requires that $\sqrt{\mu/w} \in L^1_{\mathrm{loc}}$ and \eqref{eq:nu_in_A2_means} uniformly for all $J \subseteq I$,
  but, in this generality, the derivatives in $D$ do not have the standard distributional definition.
\end{remark}

In one dimension we have the following implication. It is not clear to us if such relation between $(\mu,w)$ and $\nu$ exists in higher dimension. See \cref{thm:main_theorem1D} below.
\begin{prop}
  If $\mu,w \in A_2(I,\mathrm{d}x)$, then $\nu \in A_2(\mathbb{R},\mathrm{d}y)$, where $\mathrm{d}y = \rho'(x)\mathrm{d}x$. 
\end{prop}
\begin{proof}
  The weight $\nu$ is in $A_2(\mathbb{R},\mathrm{d}y)$ if \eqref{eq:nu_in_A2_means} holds for all $J \subset \mathbb{R}$.
  The $A_2$ condition on an interval $J$ for $\mu$ and $w$ means
  \begin{equation*}
    \int_{J} \mu(x) \D{x} \lesssim \frac{\lvert J \rvert^2}{\int_{J} 1/\mu(x) \D{x}} \quad , \quad     \int_{J} \frac{1}{w(x)} \D{x} \lesssim \frac{\lvert J \rvert^2}{\int_{J} w(x) \D{x}}.
  \end{equation*}
  Applying Cauchy--Schwarz twice gives as claimed
  \begin{align*}
    \Big(\int_J \mu \D{x}\Big)\Big( \int_J \frac{1}{w} \D{x} \Big) & \lesssim \frac{\lvert J \rvert^4}{\big(\int_J 1/\mu\big)\big( \int_J w \big)}
    \le \left( \frac{\lvert J \rvert^2}{\int_J \sqrt{w/\mu} \D{x}}\right)^2 \le \Big( \int_{J} \sqrt{\frac{\mu}{w}} \D{x} \Big)^2.
  \end{align*}
\end{proof}

\begin{example}\label{ex:power_weights}
  Consider the power weights
  $\mu(x) = x^{\alpha}$ and $w(x) = x^{-\beta}$ for $x > 0$.
  Then $\rho'(x) = \sqrt{x}^{\alpha + \beta}$ and $\nu(\rho(x)) = \sqrt{x}^{\alpha - \beta}$. %
  In computing $\rho^{-1}$,
  we distinguish three cases.
  \begin{description}
  \item[Case 1] $\alpha + \beta + 2 > 0$.
    In this case $\rho(x) = \frac{2}{\alpha + \beta +2} \sqrt{x}^{\alpha + \beta + 2}$ is strictly positive and increasing.
    Thus $\nu(y) = (\frac{\alpha + \beta + 2}{2} y)^{\frac{\alpha - \beta}{\alpha + \beta + 2}}$.
    The weight $\nu \in A_2(\mathrm{d}y)$ if and only if $-1 < \frac{\alpha - \beta}{\alpha + \beta + 2} < 1$,
    or equivalently if $\alpha > -1$ and $\beta > -1$.
    
  \item[Case 2] $\alpha + \beta + 2 < 0$.
    In this case, $\rho$ is negative
    and equals $\frac{1}{c} x^{c}$ where $c = \frac{\alpha + \beta + 2}{2} < 0$
    and  $\nu(y) = (- c y)^{\frac{\alpha - \beta}{\alpha + \beta + 2}} > 0$.
    The weight $\nu \in A_2(\mathrm{d}y)$ if and only if $-1 < \frac{\alpha - \beta}{\alpha + \beta + 2} < 1$,
    or equivalently if $\alpha < -1$ and $\beta < -1$.
    
  \item[Case 3] $\alpha + \beta = - 2$.
    In this case $\rho'(x) = 1/x$ and so $\rho(x) = \ln x$.
    Then $\rho^{-1}(y) = e^y$ and $\nu(y) = (e^y)^{(\alpha - \beta)/2}$ is in $A_2(\mathrm{d}y)$ if and only if $\alpha = \beta = -1$.
  \end{description}
  In either case $\nu \in A_2$ if and only if $\mathrm{sgn}(\alpha + 1) = \mathrm{sgn}(\beta + 1)$.
  \begin{figure}[h!]
    \begin{subfigure}{.33\textwidth}
      \begin{tikzpicture}[scale=.65]
        \draw[->] (-3, 0) -- (3, 0) node[below] {$x$};
        \draw[->] (0, -2) -- (0, 2) node[left] {$y$};
        \draw[samples=300,scale=1, domain=0:3, smooth, variable=\x, black] plot ({\x}, {sqrt(\x)});
        \draw[samples=300,scale=1, domain=-3:0, smooth, variable=\x, gray] plot ({\x}, {- sqrt(- \x)});
      \end{tikzpicture}
      \caption{Case 1.} %
    \end{subfigure}\quad
    \begin{subfigure}{.3\textwidth}      
      \begin{tikzpicture}[scale=.9]
        \draw[->] (0, 0) -- (3, 0) node[above] {$x$};
        \draw[->] (0, -2) -- (0, .5) node[left] {$y$};
        \draw[samples=100,scale=.5, domain=0.25:6, smooth, variable=\x, black] plot ({\x}, {-1/(\x)});
      \end{tikzpicture}
      \caption{Case 2.} %
    \end{subfigure}
    \begin{subfigure}{.3\textwidth}
      \begin{tikzpicture}[scale=.9]
        \draw[->] (0, 0) -- (3, 0) node[below] {$x$};
        \draw[->] (0, -1.5) -- (0, 1) node[left] {$y$};
        \draw[samples=100,scale=.5, domain=0.05:6, smooth, variable=\x, black] plot ({\x}, {ln(\x)});
      \end{tikzpicture}
      \caption{Case 3.} %
    \end{subfigure}    
    \caption{Completeness of the $y$-axes.
      In Case 1, $\rho(x) = \sqrt{x}$ on $\mathbb{R}_+$ can be extended to an odd bijection $\mathbb{R} \to \mathbb{R}$.
      In Case 2, $\rho(x) = -1/x$ is not surjective onto $\mathbb{R}$.
      In Case 3, $\rho(x) = \ln(x)$ is a bijection from $\mathbb{R}_+$ to $\mathbb{R}$.}
    \label{fig:1dim_examples}
  \end{figure}
  Case 2 shows that it is possible that $\nu \in A_2$ even if $\mu$ and $w$ are not.
  Note that in the extension of Case 1 to an odd bijection, and in Case 3, the map $\rho$ is a bijection and maps onto a complete manifold,
  while in Case 2 the map $\rho$ is not surjective. See \cref{fig:1dim_examples}. %

  Assuming that $\lvert \alpha \rvert,\lvert \beta\rvert < 1$ and extending to power weights
  $\mu(x) = \lvert x \rvert^{\alpha}$ and $ w(x) = \lvert x \rvert^{-\beta}$,
  \cref{thm:main_theorem1D} applies and gives
  quadratic estimates for the operator $B D$, where
  \begin{equation*}
    D =
    \begin{bmatrix}
      0 & - \lvert x \rvert^{-\alpha} \partial_x\lvert x \rvert^{-\beta} \\
      \partial_x& 0
    \end{bmatrix},
  \end{equation*}
  on the weighted space $L^2(\mathbb{R}, \lvert x \rvert^{\alpha}) \oplus L^2(\mathbb{R}, \lvert x \rvert^{-\beta})$.
\end{example}

\begin{cor}\label{cor:Kato_weighted1d}
  Let $I \subseteq \mathbb{R}$ and let $\mu,w \in A_2^{\mathrm{loc}}(\mathbb{R})$
  satisfy the assumptions of \cref{thm:main_theorem1D}.
  In particular $\nu \in A_2(\mathbb{R},\mathrm{d}y)$.
  Let $a,b$ be two complex-valued functions on $I$ such that
  \begin{equation}
    \label{eq:hyp_cor_Kato_1d}
    \begin{aligned}
      \mu(x) & \lesssim \Real \, a(x) \, , \,  \lvert a(x) \rvert \lesssim \mu(x) \, , \\
      w(x) & \lesssim \Real \, b(x) \, , \,  \lvert b(x) \rvert \lesssim w(x) 
    \end{aligned}
\end{equation}
for a.e. $x \in I$.
  Then the following Kato square root estimate holds:
  \begin{equation*}
    \lVert \sqrt{- (1/a) \partial_x b \partial_x} u \rVert_{L^2(I,\mu)} \eqsim \lVert \partial_x u \rVert_{L^2(I,w)} .
  \end{equation*}  
\end{cor}

\begin{proof}
  Consider the multiplication operator
  $B = \Big [
  \begin{smallmatrix}
    \mu/a & 0 \\
    0 & b/w
  \end{smallmatrix} \Big ]
  $.
  The hypothesis in \eqref{eq:hyp_cor_Kato_1d} implies that $B$ is bounded and accretive.
  Since
  \begin{equation*}
    B =
    \begin{bmatrix}
      \sqrt{\mu} & 0 \\
      0 & \sqrt{w}
    \end{bmatrix}
    B
    \begin{bmatrix}
      \sqrt{\mu} & 0 \\
      0 & \sqrt{w}
    \end{bmatrix}^{-1}    
  \end{equation*}
  holds for any diagonal matrix $B$,
  then $B$ is
  $\begin{bsmallmatrix}
    \mu & \\
        & w
  \end{bsmallmatrix}
  $-bounded and $
  \begin{bsmallmatrix}
    \mu & \\
        & w
  \end{bsmallmatrix}
  $-accretive.  %
  The desired estimate follows by applying
  \cref{thm:main_theorem1D} to $B$ and
  $D$ as defined in \eqref{eq:D_in_1D}.
  Indeed, the perturbed operator $B D$ equals
  \begin{equation*}
    B D =
    \begin{bmatrix}
      0 & - (1/a) \partial_x w \\
      \frac{b}{w} \partial_x & 0 
    \end{bmatrix}
  \end{equation*} and so
  \begin{equation*}
    \lVert \sqrt{- (1/a) \partial_x b \partial_x} u \rVert_{L^2(I,\mu)} =
    \big\lVert  \sqrt{(B D)^2}\big[
    \begin{smallmatrix}
      u \\ 0
    \end{smallmatrix}
    \big] \big\rVert_{\mathcalboondox{H}}.
  \end{equation*}
  The boundedness of the $H^\infty$ functional calculus for $B D$ on $\mathcalboondox{H} = L^2(I,\mu) \oplus L^2(I,w)$ implies
  that $\mathrm{sgn}(B D)$ is a bounded and invertible operator on $\mathcalboondox{H}$.
  Since $\sqrt{(B D)^2} = \mathrm{sgn}(B D)B D$, we have
  \begin{equation*}   
    \big\lVert  \sqrt{(B D)^2}\big[
    \begin{smallmatrix}
      u \\ 0
    \end{smallmatrix}
    \big] \big\rVert_{\mathcalboondox{H}} \eqsim \big\lVert B D \big[
    \begin{smallmatrix}
      u \\ 0
    \end{smallmatrix}
    \big] \big\rVert_{\mathcalboondox{H}} \eqsim \big\lVert D \big[
    \begin{smallmatrix}
      u \\ 0 
    \end{smallmatrix}
    \big] \big\rVert_{\mathcalboondox{H}} \eqsim \lVert \partial_x u \rVert_{L^2(I,w)}.
  \end{equation*}
\end{proof}

In one dimension, it is well known from \cite{CMcM82} %
that the Kato square root estimate for uniformly bounded and accretive coefficients is equivalent to the $L^2$ boundedness of the Cauchy singular integral on Lipschitz curves.
It is therefore natural to investigate what implications \cref{cor:Kato_weighted1d} has for the boundedness of the Cauchy singular integral.
Although the following two examples do not give any new result, we include them since the observations may be useful in future work. 

\begin{example}[Cauchy integral on rectifiable graphs]
  Consider a curve $\gamma \coloneqq (t, \varphi(t))$ as the graph of a function $\varphi \colon \mathbb{R} \to \mathbb{R}$.
  The curve $\gamma$ is Lipschitz if and only if $\varphi' \in L^\infty$.

  The Cauchy singular integral
  \begin{equation*}
    \mathscr{C}_\gamma(x) \coloneqq \frac{i}{\pi} \mathrm{ p.v.} \int_{-\infty}^\infty \frac{u(y)}{y + i \varphi(y) - (x + i\varphi(x))} (1 + i \varphi'(y)) \D{y}
  \end{equation*}               %
  and its boundedness on $L^2(\gamma)$ for Lipschitz curves is a classical and famous problem in analysis.
It was first showed by Calderón \cite{zbMATH03581044} %
that $\mathscr{C}_\gamma \colon L^2(\gamma) \to L^2(\gamma)$ for a curve $\gamma \subseteq \mathbb{C}$ with small Lipschitz constant $\lVert \varphi' \rVert_{L^\infty}$. 
This smallness assumption was removed by Coifman--M\textsuperscript{c}Intosh--Meyer in \cite{CMcM82}, %
where only $\lVert \varphi' \rVert_{L^\infty} < \infty$ was assumed.
and finally David \cite{zbMATH03853693} %
showed that $\mathscr{C}_\gamma$ is bounded on $L^2(\gamma)$
if and only if the curve $\gamma$ is Ahlfors--David-regular,
meaning that the 1-dimensional Hausdorff measure $\mathscr{H}^1$ restricted on the curve satisfies
\begin{equation*}
  \mathscr{H}^1(\gamma \cap B(x,r)) \eqsim r
\end{equation*}
for any ball $B(x,r)$ centred at $x \in \gamma$.
A crucial observation due to Alan M\textsuperscript{c}Intosh
which led to the seminal work \cite{CMcM82}
is that the Kato estimate %
\begin{equation*}
  \lVert \sqrt{- (1/a) \partial_x b \partial_x} u \rVert_{L^2(\mathbb{R})} \eqsim \lVert \partial_x u \rVert_{L^2(\mathbb{R})},
\end{equation*}
for $b = 1/a$, implies the $L^2$-estimate for $\mathscr{C}_\gamma$ on Lipschitz curves.
See also Kenig and Meyer \cite{zbMATH04044540}. %

One can ask if the weighted estimates in \cref{cor:Kato_weighted1d}
can be used to prove that $\mathscr{C}_\gamma$ is bounded on  Ahlfors--David-regular graphs
more general than Lipschitz graphs.
This is still unclear to us.
The natural strategy is as follows.
As in \cite{McIQ91},
the Cauchy singular integral
can be written as $\mathrm{sgn}((1/a(x)) i \partial_x )$, for multiplier $a(x) = 1 + i \varphi'(x)$,
see also \cite[Consequence 3.2]{AKM^c}.
Note that the arclength measure on $\gamma$ is $\mathrm{d}s \coloneqq \sqrt{1 + (\varphi')^2} \D{x} = \mu \D{x}$.
Boundedness of $\mathscr{C}_\gamma$ in $L^2(\gamma, \mathrm{d}s)$ thus amounts to
\begin{equation*}
  \lVert \mathrm{sgn}\big( (1/a) i \partial_x \big) u \rVert_{L^2(\mathbb{R},\mu)} \lesssim \lVert u \rVert_{L^2(\mathbb{R},\mu)} .
\end{equation*}
By functional calculus, this is equivalent to
\begin{equation*}
  \lVert \sqrt{- (1/a) \partial_x (1/a) \partial_x } u \rVert_{L^2(\mathbb{R},\mu)} \lesssim \lVert (1/a) \partial_x u \rVert_{L^2(\mathbb{R},\mu)} = \lVert \partial_x u \rVert_{L^2(1/\mu)} .
\end{equation*}
The latter estimate would follow from \cref{cor:Kato_weighted1d}
with $b = 1/a$, $w = 1/\mu$, if the hypotheses were satisfied,
since in this case $\sqrt{\mu/w} = \mu = \sqrt{1+(\varphi')^2}$ and $\nu(y) = \sqrt{\mu w} = 1$.
However \cref{cor:Kato_weighted1d} does not apply here,
since the accretivity condition $\Real \, a(x) = 1 \gtrsim  \mu(x)$ is not satisfied,
unless $\varphi'$ is bounded.
\end{example}

We end this section by noting that the matrix-weighted Kato square root estimate \eqref{eq:matrix_weighted_Kato}
which we consider in this paper, despite looking like a two-weight estimate,
should be seen as a one-weight estimate,
as the proof of \cref{thm:main_theorem1D} clearly shows.
In the following example we see that our results apply
only when the weights in the square root operator correctly match the weights in the norms.

\begin{example}[Two-weight Hilbert transform]
  Consider the two-weight estimate
  \begin{equation}\label{eq:two-weight-Hilbert}
    \lVert H u \rVert_{L^2(\mu)} \lesssim \lVert u \rVert_{L^2(w)}
  \end{equation}
  for the Hilbert transform
  \begin{equation*}
    H u (x) \coloneqq \frac{i}{\pi} \mathrm{p.v.} \int_{\mathbb{R}} \frac{u(y)}{y-x} \D{y} .
  \end{equation*} 
  The problem of characterising for which weights $\mu,w$
  the estimate \eqref{eq:two-weight-Hilbert} holds was solved in \cite{MR3285857}. %
  If we use functional calculus
  to write $H$ as $\sqrt{-\partial_x^2} (i\partial_x)^{-1}$,
  then \eqref{eq:two-weight-Hilbert} amounts to
  \begin{equation}\label{eq:two-weight-Hilbert_func_calc}
    \lVert \sqrt{-\partial_x^2} u \rVert_{L^2(\mu)} \lesssim \lVert \partial_x u \rVert_{L^2(w)} .
  \end{equation}
  Changing variables $y = \rho(x)$ and $u(x) = v(\rho(x))$ as in \cref{lemma:rubber_band},
  and using the chain rule: $\partial_x = \rho' \partial_y$,
  the two-weight estimate \eqref{eq:two-weight-Hilbert} becomes
  \begin{equation*}
    \int \big\lvert \sqrt{-(\rho' \partial_y)^2} v(\rho(x)) \big\rvert^2 \mu(x) \D{x} \lesssim \int \lvert (\rho'(x) \partial_y v)(\rho(x)) \rvert^2 w(x) \D{x} .
  \end{equation*}
  Choosing $\rho'(x) = \sqrt{\mu(x)/w(x)}$ gives $(\rho')^2 w = \mu$ in the right hand side. %
  Changing variables and using $\nu\circ\rho  = \sqrt{w \mu}$ yields  
  \begin{equation*}
    \mu(x) \D{x} = \sqrt{\mu(x) w(x)} \cdot \sqrt{\mu(x)/w(x)} \D{x}  = (\nu \circ \rho)(x) \cdot \rho'(x) \D{x} = \nu(y) \D{y}.
  \end{equation*}
  Thus estimate \eqref{eq:two-weight-Hilbert} holds if and only if the one-weight estimate
  \begin{equation}\label{eq:other_form_of_two-weight_Hilbert}
    \big\lVert \sqrt{-(\lambda \partial_y)^2} v \big\rVert_{L^2(\nu)} \lesssim \lVert \partial_y v \rVert_{L^2(\nu)} 
  \end{equation}
  holds with the weight $\lambda(y) \coloneqq \rho'(\rho^{-1}(y)) = \sqrt{\mu(\rho^{-1}(y))/w(\rho^{-1}(y))}$
  in the Kato square root operator.
  \cref{cor:Kato_weighted1d} does not apply directly to \eqref{eq:other_form_of_two-weight_Hilbert},
  nor to \eqref{eq:two-weight-Hilbert_func_calc},
  since it requires that the weights in the Kato square root operator correctly match the weights in the norms.

\end{example}

\section{The \texorpdfstring{$(\mu, W)$}{(μ,W)} manifold $M$}
\label{sec:higher_dim}
We now seek to generalise the results in \cref{sec:1dim}
to higher dimension $d \ge 2$, starting with \cref{lemma:rubber_band}.
To cover general matrix weights $W$, we need to allow for more general diffeomorphisms
$\rho \colon \Omega \subseteq \mathbb{R}^d \to M$,
where now $M$ is some auxiliary smooth $d$-dimensional Riemannian manifold,
and $\Omega$ an open set in $\mathbb{R}^d$.
The metric $g$ for $M$ will be determined by $\mu$ and $W$,
but not the differential structure on $M$.
In general, smooth weights $(\mu, W)$ will define a metric $g$ for a manifold with non-zero curvature.
For this reason, we need to allow for curved manifolds.
Also, we will soon work with homeomorphisms $\rho$ which are not smooth.
So, while as sets and topological spaces $M$ and $\Omega \subseteq \mathbb{R}^d$ can be identified,
their differential structures will differ.
A manifestation of this is that the metric $g$ on $M$ will be smooth with respect to the differential structure on $M$,
but not with respect to the one on $\mathbb{R}^d$.
The natural pullback generalising \eqref{eq:rubber_band_pullback}
for the differential operator $D$ in \eqref{eq:first_appearance_of_D} is now
\begin{equation}\Tiso \colon 
  \begin{bmatrix}
    v_1(y) \\
    v_2(y) 
  \end{bmatrix} \mapsto
  \begin{bmatrix}
    v_1(\rho(x)) \\
    (\mathrm{d}\rho_x)^\star v_2(\rho(x))
  \end{bmatrix}
  \eqqcolon
  \begin{bmatrix}
    u_1(x) \\
    u_2(x)
  \end{bmatrix}
  . \label{eq:definition_rubber_band_higher_dim}
\end{equation}
Here $v_1 \colon M \to \mathbb{C}$ is a scalar function on $M$
and $v_2$ is a section of the cotangent bundle $T^{\star}M$, which we identify with $TM$ using the metric $g$.
This is important because,
although we can view $v_2$ as a vector on $M$, it is a 1-form, so its pullback
is obtained by multiplying $v_2 \circ \rho$ by the transpose $(\mathrm{d}\rho)^\star$
of the Jacobian matrix $\mathrm{d}\rho$.
Below $J_\rho$ denotes the determinant of the Jacobian matrix $\det(\mathrm{d}\rho) \coloneqq \det( g )^{1/2}$,
where $g = (\mathrm{d}\rho)^\star \mathrm{d}\rho$ is the Riemannian metric on $M$ pulled back to $\mathbb{R}^d$. 

  Here and below, to ease notation, we shall identify maps defined on $\mathbb{R}^d$
  and on $M$ through $\rho$, writing for example $\nabla_M v_1$ for $(\nabla_M v_1) \circ \rho$.
  We use $v(y)$ for functions defined on $M$ and $u(x)$ for functions defined on $\mathbb{R}^d$.
  With a slight abuse of notation,
  we use the abbreviations $J_\rho(y)$,$\D{\rho}_y$ and $u(y)$
  for $J_\rho(\rho^{-1}(y))$, $\D{\rho}_{\rho^{-1}(y)}$, $u(\rho^{-1}(y))$.
  The differential operators $\nabla$ and $\mathrm{div}$ are always defined on $\mathbb{R}^d$.

To write the operator $D_M$ similar to $D$ %
we need the chain rule:
\begin{equation*}
  \nabla u_1 = (\mathrm{d}\rho)^\star \nabla_M v_1,
\end{equation*}
which holds in the weak sense by \cref{prop:chain_rule_scalar}.
We also require the $L^2$-adjoint result for vector fields $u_2 \colon \mathbb{R}^d \to \mathbb{C}^d$
in \cref{prop:chain_rule_vector}.
We compute
\begin{align}
  \Tiso^{-1} D \Tiso
  \begin{bmatrix}
    v_1 \\ v_2
  \end{bmatrix}
  & = \Tiso^{-1} D
    \begin{bmatrix}
      v_1 \circ \rho \\
      (\mathrm{d}\rho_x^\star v_2) \circ \rho
    \end{bmatrix}
    = \Tiso^{-1} 
    \begin{bmatrix}
      - (1/\mu) \mathrm{div} W [ (\mathrm{d}\rho_x^\star v_2) \circ \rho ] \\
      \nabla (v_1 \circ \rho)    
    \end{bmatrix} \nonumber \\
  & = \Tiso^{-1} 
    \begin{bmatrix}
      - (1/\mu) \mathrm{div}_{M} \Big\{ J_\rho^{-1} \D{\rho}_x \big( W (\mathrm{d}\rho_x)^\star v_2 \big) \Big\} J_\rho \\ %
      \nabla u_1 
    \end{bmatrix} \nonumber \\
  & = 
    \begin{bmatrix}
      - (1/\mu) J_\rho \mathrm{div}_{M} \Big\{ J_\rho^{-1} \D{\rho}_y \big( W (\mathrm{d}\rho_y)^\star v_2 \big) \Big\}  \\ %
      \nabla_M v_1 
    \end{bmatrix}. \label{eq:we_demand_condition_later}
\end{align}

We obtain the following generalisation of \cref{lemma:rubber_band}.
\begin{lemma}\label{lemma:rubber_band_higher_dim}
  Assume that $\mu$ is a scalar weight on $\mathbb{R}^d$ and that $W$ is a matrix weight on $\mathbb{R}^d$.
  Assume that $\mu$ and $W$ are smooth around $\rho(x_0) \in \mathbb{R}^d$.
  Set
  \begin{equation*}
    g \coloneqq \mu W^{-1} \quad \text{ and } \quad \nu \coloneqq \mu / \sqrt{\det g} . %
  \end{equation*}
  Let $M$ be a Riemannian manifold with chart $(U,\rho)$ around $x_0$
  and metric $g$ in this chart.
  Let
  \begin{equation}\label{eq:D_M}
    D_M \coloneqq
    \begin{bmatrix}
      0 & - (1/\nu) \mathrm{div}_M \nu \\
      \nabla_M & 0
    \end{bmatrix}.
  \end{equation}
  Then $\Tiso \colon L^2(U,\nu) \oplus L^2(TU,\nu I) \to L^2(\rho^{-1}(U), \mu) \oplus L^2(\rho^{-1}(U);\mathbb{C}^d, W)$
  defined in \eqref{eq:definition_rubber_band_higher_dim},
  is an isometry, and $\Tiso^{-1} D \Tiso = D_M$.
\end{lemma}

\begin{remark}
  There is a one-to-one correspondence between
  the pairs of weights $(\mu,W)$ and 
  the pairs $(g,\nu)$ of Riemannian metric and weight,
  since inversely $\mu = \nu \sqrt{\det g}$, and $W = (\nu \sqrt{\det g}) g^{-1}$. %
\end{remark}

\begin{proof}[Proof of \cref{lemma:rubber_band_higher_dim}]
  To obtain the operator $D_M$ with a single scalar weight $\nu$ on a manifold,
  in \eqref{eq:we_demand_condition_later} we require that
  \begin{equation*}
    (1/\mu) J_\rho = 1/\nu \quad \text{ and } \quad J_\rho^{-1} \D{\rho} \, W \D{\rho}^\star = \nu I ,
  \end{equation*}
  where $I$ is the identity matrix.
  The first condition yields $\mu = J_\rho \nu$.
  Since the volume change is $J_\rho = \sqrt{\det g}$,
  we have $\nu = \mu / \sqrt{\det g}$ as stated.
  For the second one,
  since the metric in a chart $\rho$ is $g = \mathrm{d}\rho^\star \mathrm{d}\rho$,
  and the matrices $\mathrm{d}\rho$ and $\mathrm{d}\rho^\star$ commute with the scalars $\nu$ and $J_\rho$, we have
  \begin{equation*}
    \frac{W}{ J_\rho \nu } = \D{\rho}^{-1} (\D{\rho}^\star)^{-1} = \big( \mathrm{d}\rho^\star \mathrm{d}\rho \big)^{-1} = g^{-1} ,
  \end{equation*}
  and so $g = \mu W^{-1}$.
  To see that the map $\Tiso$ in \eqref{eq:definition_rubber_band_higher_dim} is an isometry,
  it is enough to compute
  \begin{equation}\label{eq:check_isometry_on_scalars}
    \int_{\mathbb{R}^d} \lvert u_1(x) \rvert^2 \mu(x) \D{x} = \int_{M} \lvert v_1(y) \rvert^2 \underbrace{\frac{\mu}{\sqrt{\det g}}(y)}_{= \, \nu(y)} \D{y} %
  \end{equation}
  where $\mathrm{d}y$ is the Riemannian measure on $M$.
  Also                          %
  \begin{equation}\label{eq:check_isometry_on_vectors}
    \begin{aligned}
      \int_{\mathbb{R}^d} \langle W(x) u_2(x), u_2(x) \rangle\D{x} & =  \int_{\mathbb{R}^d} \langle W(x) (\mathrm{d}\rho_x)^\star v_2(\rho(x)), (\mathrm{d}\rho_x)^\star v_2(\rho(x)) \rangle \D{x} \\
      & = \int_M \langle W \mathrm{d}\rho^\star v_2(y), \mathrm{d}\rho^\star v_2(y) \rangle \frac{\D{y}}{\sqrt{\det g}} \\
      & = \int_M \langle \frac{1}{\sqrt{\det g}} \D{\rho} W \mathrm{d}\rho^\star v_2(y), v_2(y) \rangle \D{y} \\
      & = \int_M \lvert v_2(y) \rvert^2 \nu(y) \D{y} .
    \end{aligned}
  \end{equation}
  This concludes the proof.
\end{proof}
We aim to prove a matrix-weighted Kato square root estimate on $\Omega \subseteq \mathbb{R}^d$,
by applying \cite[Theorem 1.1]{AMR} to the one-scalar-weight operator
$D_M$ on $M$ in \eqref{eq:D_M}
and pulling back the result to $\mathbb{R}^d$.
However, this requires a modification of \cref{lemma:rubber_band_higher_dim}
since \cite[Theorem 1.1 and Theorem 1.2]{AMR} only apply to prove inhomogeneous Kato square root estimates, %
since only local square function estimates can be proved on $M$ without further hypothesis on its geometry at infinity.
As in \cite[Eq. (2.4)]{AMR} we introduce inhomogeneous first-order differential operators
\begin{align}
  \widetilde{D} & = \label{eq:inhomog_D}
  \begin{bmatrix}
    0 & I & - (1/\mu) \mathrm{div} W \\
    I & 0 & 0 \\
    \nabla & 0 & 0
  \end{bmatrix} \, \text{ acting on } \quad
                  \widetilde{\mathcalboondox{H}} \coloneqq
              \begin{bmatrix}
                L^2(\Omega, \mu) \\
                L^2(\Omega, \mu) \\
                L^2(\Omega; \mathbb{C}^d , W)
              \end{bmatrix} , \\
  \widetilde{D}_M & = \label{eq:inhomog_D_M}
  \begin{bmatrix}
    0 & I & - (1/\nu) \mathrm{div}_{M} \nu \\
    I & 0 & 0 \\
    \nabla_{M} & 0 & 0
  \end{bmatrix} \, \text{ acting on }\quad
                   \widetilde{\mathcalboondox{H}}_M \coloneqq
              \begin{bmatrix}
                L^2(M, \nu) \\
                L^2(M, \nu) \\
                L^2(TM,\nu)
              \end{bmatrix} ,
\end{align}
where divergence and $\nabla$ in \eqref{eq:inhomog_D} are on $\mathbb{R}^d$,
and the square brackets denote the sum of spaces.
The domains of the operators $\nabla$ and $\nabla_M$ are weighted Sobolev spaces
\begin{align*}
  \mathcal{H}^1_{\mu,W}(\Omega) & \coloneqq \big\{ f \in W^{1,1}_{\mathrm{loc}}(\Omega), f \in L^2_{\mathrm{loc}}(\Omega,\mu) \text{ with } \nabla f \in L^2_{\mathrm{loc}}(\Omega;\mathbb{R}^d,W) \big\}, \\
  \mathcal{H}^1_{\nu}(M) & \coloneqq \big\{ f \in W^{1,1}_{\mathrm{loc}}(M), f \in L^2_{\mathrm{loc}}(M,\nu) \text{ with } \nabla_{M} f \in L^2_{\mathrm{loc}}(T M,\nu I) \big\}
\end{align*}
respectively, so $\mathsf{dom}(\nabla) = \mathcal{H}^1_{\mu,W}(\Omega)$ and $\mathsf{dom}(\nabla_M) = \mathcal{H}^1_{\nu}(M) $.
The closed operator $-\mathrm{div}$ with domain
\begin{equation*}
  \mathsf{dom}(\mathrm{div}) = \big\{ h \in L^2_{\mathrm{loc}}(\Omega;\mathbb{R}^d,W^{-1}) \,\colon\, \mathrm{div}h \in L^2_{\mathrm{loc}}(\Omega,\mu^{-1}) \big\}
\end{equation*}
is the adjoint of $\nabla$ with respect the unweighted $L^2$ pairing.
In the same way, $-\mathrm{div}_M$ is the closed operator with domain
\begin{equation*}
  \mathsf{dom}(\mathrm{div}_M) = \big\{ h \in L^2_{\mathrm{loc}}(TM,\nu^{-1}) \,\colon\, \mathrm{div}_Mh \in L^2_{\mathrm{loc}}(M,\nu^{-1}) \big\}
\end{equation*}
and it is the adjoint of $\nabla_M$ with respect the unweighted $L^2$ pairing on $M$.

In \cref{lemma:rubber_band_higher_dim} we assumed qualitatively that $\rho$ was a smooth diffeomorphism.
The following results in this section we relax this condition to a Sobolev $W^{1,1}$ regularity, which suffices for the proof.
Already in one dimension, we have seen the usefulness of such weaker regularity assumption in \cref{ex:power_weights}. %

Consider the pullback
$\widetilde{\Tiso} \colon \widetilde{\mathcalboondox{H}}_M \to \widetilde{\mathcalboondox{H}}$
via $\rho \in W^{1,1}_{\mathrm{loc}}$ given by
\begin{equation*}
  \widetilde{\Tiso} \colon
  \begin{bmatrix}
    v_1(y) \\ v_0(y) \\ v_2(y)
  \end{bmatrix} \mapsto
  \begin{bmatrix}
    v_1 \circ \rho \\
    v_0 \circ \rho \\
    \D{\rho}^\star v_2 \circ \rho
  \end{bmatrix} \eqqcolon
  \begin{bmatrix}
    u_1(x) \\ u_0(x) \\ u_2(x)
  \end{bmatrix}. 
\end{equation*}
The map $\widetilde{\Tiso}$ preserves the domains of the operators $\widetilde{D}$ and $\widetilde{D}_M$.
\begin{lemma}
  The map $\widetilde{\Tiso}$ is an isometry and %
   $\widetilde{\Tiso}( \mathsf{dom}(\widetilde{D}_M) ) = \mathsf{dom}(\widetilde{D})$.
\end{lemma}

\begin{proof}
  For scalar-valued functions,
  apply \cref{prop:chain_rule_scalar} with $v = \nu$ and $V = \nu I$.
  Note that since $\nu\circ\rho = \mu/J_\rho$, then $v_\rho = \mu$.
  Also, since the metric $\mathrm{d}\rho^{-1}(\mathrm{d}\rho^{-1})^\star = g^{-1} = \mu^{-1} W$,
  it follows that $V_\rho = W$.  
  For vector fields,
  if $\vec{u} \in L^2(\Omega;\mathbb{R}^d,W^{-1})$ with $\mathrm{div} \vec{u}$ in $L^2(\Omega,\mu^{-1})$,
  apply  \cref{prop:chain_rule_vector} with $V = W^{-1}$ and $v = \mu^{-1}$.
  Indeed, $V^{\rho} = \nu^{-1} I$ and $v^{\rho} = \nu^{-1}$,
  so $ J_\rho^{-1} \rho_* \vec{u} = J_\rho^{-1} \mathrm{d}\rho \vec{u} \circ \rho^{-1} \in L^2(TM, \nu^{-1} I) $ and
  \begin{equation*}
    \mathrm{div}\Big( \frac{\rho_*}{J_\rho} \vec{u} \Big) = \frac{\rho_*}{J_\rho}( \mathrm{div} \vec{u} ) \in L^2(M,\nu^{-1}).
  \end{equation*}
\end{proof}
As in the proof of \cref{lemma:rubber_band_higher_dim},
one sees that $\widetilde{\Tiso}$ is an isometry. 
A calculation as in \eqref{eq:we_demand_condition_later} shows that
\begin{equation*}
  \widetilde{\Tiso}^{-1} \widetilde{D} \widetilde{\Tiso} = \widetilde{D}_M  .
\end{equation*}

We have the following generalisation of \cref{thm:main_theorem1D}.
\begin{theorem}\label{thm:main_theorem_higher_dim}
  Let $ \Omega \subseteq \mathbb{R}^d $ be an open set,
  and let $\rho \colon \Omega \to M$ be a $W^{1,1}_{\mathrm{loc}}$ homeomorphism onto a complete, smooth Riemannian manifold $(M,h)$.
  Let $\mu, W$ be scalar and matrix weights in $A_2^{\mathrm{loc}}(\Omega)$. %
  Assume that the metric on $M$ pulled back via $\rho$ is
  \begin{equation*}
    g = \mu W^{-1}
  \end{equation*}
  and define the scalar weight $\nu = \mu / \sqrt{\det g}$ on $M$. %
  Let $\widetilde{D}$ be the differential operator in \eqref{eq:inhomog_D}
  and let $\widetilde{B}$
  be a $(\mu \oplus \mu \oplus W)$-bounded, $(\mu \oplus \mu \oplus W)$-accretive multiplication operator on $\widetilde{\mathcalboondox{H}}$
  as in \cref{def:W-boundedness_and_accretivity}.
  If the manifold $M$ has Ricci curvature bounded from below and positive injectivity radius,
  and if $\nu \in A_2^{R}(M)$, for some $R>0$,
  then $\widetilde{B} \widetilde{D}$ and $\widetilde{D} \widetilde{B}$ are bisectorial operators that
  satisfy quadratic estimates and have bounded $H^\infty$ functional calculus in $\widetilde{\mathcalboondox{H}}$.
\end{theorem}

\begin{remark} %
  The Riemannian manifold $M$ is assumed to be smooth
  with smooth metric.
  But since the map $\rho$ is not smooth in general,
  the pullback $g$ of the smooth metric of $M$
  on $\Omega$ may be non-smooth.
  See \cref{fig:diagram-non-smooth-structure}.
\end{remark}

\begin{figure}[h]
  \begin{tikzcd}%
    & M \\ %
    \Omega \subseteq \mathbb{R}^d \arrow[ru, "\rho"] \arrow[r, "\varphi^{-1} \circ \rho"] & \mathbb{R}^d \arrow[u, "\varphi"]
  \end{tikzcd}
  \caption{The Riemannian manifold $M$ with a chart $\varphi$ from its smooth atlas.
    A function $f$ on $M$ is smooth if $f \circ \varphi$ is smooth.
    But $f \circ \rho$ is not in general smooth
    since the map $\varphi^{-1} \circ \rho$ is only in $W^{1,1}$.}
  \label{fig:diagram-non-smooth-structure}
\end{figure}

\begin{proof}[Proof of \cref{thm:main_theorem_higher_dim}]
  Given the differential operator $\widetilde{D}$ as in \eqref{eq:inhomog_D},
  consider the operators $\widetilde{D}_M \coloneqq \widetilde{\Tiso}^{-1} \widetilde{D} \widetilde{\Tiso}$ given in \eqref{eq:inhomog_D_M}
  and the operator $\widetilde{B}_M \coloneqq \widetilde{\Tiso}^{-1} \widetilde{B} \widetilde{\Tiso}$.

  \cref{lemma:rubber_band_higher_dim} shows that the extended pullback transformation $\widetilde{\Tiso}$
  is an isometry between the weighted spaces $\widetilde{\mathcalboondox{H}}_M$ and $\widetilde{\mathcalboondox{H}}$.
  Indeed, let $ u = \widetilde{\Tiso} v$, then
  \begin{equation*}
    \langle \widetilde{\Tiso}^{-1}(\widetilde{B} u) , \widetilde{\Tiso}^{-1}(\widetilde{B} u) \rangle_{\widetilde{\mathcalboondox{H}}_M}
    = \langle \widetilde{\Tiso}\widetilde{\Tiso}^{-1}(\widetilde{B} u), \widetilde{B} u \rangle_{\widetilde{\mathcalboondox{H}}}
  \end{equation*}
  from which follows that $\widetilde{B}_M$ is
  $(\nu \oplus \nu \oplus \nu I)$-bounded, and $(\nu \oplus \nu \oplus \nu I)$-accretive if and only if
  $\widetilde{B}$ is $(\mu \oplus \mu \oplus W)$-bounded, $(\mu \oplus \mu \oplus W)$-accretive.

  The result \cite[Lemma 2.3]{AMR} implies that $\widetilde{D}_M$ is self-adjoint,
  and so
  is the operator $\widetilde{D} = \widetilde{\Tiso} \widetilde{D}_M \widetilde{\Tiso}^{-1}$, since $\widetilde{\Tiso}$ is unitary.
  By \cite[Theorem 1.1]{AMR},
  the operator $\widetilde{B}_M \widetilde{D}_M$ has bounded $H^\infty$ functional calculus in $L^2(M; \mathbb{C}^d \oplus TM, \nu I)$.
  The same holds for the operator $\widetilde{B}\widetilde{D}$ via the isometry $\widetilde{\Tiso}$,
  and for $\widetilde{D}\widetilde{B} = \widetilde{B}^{-1}(\widetilde{B}\widetilde{D})\widetilde{B}$.
\end{proof}

Analogous to \cref{cor:Kato_weighted1d},
we derive from \cref{thm:main_theorem_higher_dim} the following Kato square root estimate.

\begin{cor}
  Assume that $\rho \colon \Omega \to M$, $\mu, W$, $g,\nu$ satisfy the hypotheses of \cref{thm:main_theorem_higher_dim}.
  Consider the operator
  \begin{equation*}
    Lu \coloneqq -\frac{1}{\mu}\mathrm{div}A\nabla u -\frac{1}{\mu}\mathrm{div}(\vec{b} u) + \frac{1}{\mu}\langle\vec{c},\nabla u\rangle+ d \cdot u ,
  \end{equation*}
  where the %
  matrix 
  \begin{equation*} B \coloneqq
    \begin{bmatrix}      
       d & \mu^{-1/2}\vec{c} \, W^{-1/2} \\
       W^{-1/2}\vec{b} \mu^{-1/2} & W^{-1/2}AW^{-1/2}
    \end{bmatrix}
  \end{equation*}
  is bounded and accretive with respect to the Euclidean metric, meaning that
  \begin{equation*}
    B \in L^\infty \quad \text{ and } \quad \inf_{\substack{x \in \Omega \\ v \in \mathbb{C}^d}} \frac{\Real\langle B(x) v , v \rangle}{ \lvert v \rvert^2} \gtrsim 0.
  \end{equation*}
  Then the following Kato square root estimate 
  \begin{equation*}
    \lVert \sqrt{a L} u \rVert_{L^2(\Omega,\mu)} \eqsim \lVert \nabla u \rVert_{L^2(\Omega;\mathbb{C}^d,W)} + \lVert u \rVert_{L^2(\Omega,\mu)}
  \end{equation*}
  holds for any complex-valued function $a \in L^\infty(\Omega)$ such that $\inf_{\Omega} \Real( a )\gtrsim 1$.
\end{cor}
\begin{proof}
  Apply \cref{thm:main_theorem_higher_dim} to $\widetilde{D}$ defined in \eqref{eq:inhomog_D}
  and coefficients
  \begin{equation*}
    \widetilde{B} =
    \begin{bmatrix}
      a & 0 & 0 \\
      0 & d & \mu^{-1}\vec{c} \\
      0 & W^{-1}\vec{b} & W^{-1}A
    \end{bmatrix}.
  \end{equation*}
  By the hypothesis on the coefficient
  and the property of $a$, the matrix $\widetilde{B}$ is $(\mu \oplus \mu \oplus W)$-bounded and $(\mu \oplus \mu \oplus W)$-accretive, see \cref{def:W-boundedness_and_accretivity}.
  By \cref{thm:main_theorem_higher_dim} the operator $\widetilde{B}\widetilde{D}$ 
  has bounded $H^\infty$ functional calculus on  $\widetilde{\mathcalboondox{H}} = L^2(\Omega,\mu)^2 \oplus L^2(\Omega;\mathbb{C}^d,W)$.
  This implies the boundedness and invertibility of the operator $\mathrm{sgn}(\widetilde{B}\widetilde{D})$,
  and so by writing $\sqrt{(\widetilde{B}\widetilde{D})^2} = \mathrm{sgn}(\widetilde{B}\widetilde{D})\widetilde{B}\widetilde{D}$ we have
  \begin{equation*}
   \Big\lVert  \sqrt{(\widetilde{B}\widetilde{D})^2}\Big[
    \begin{smallmatrix}
      u \\ 0 \\ 0
    \end{smallmatrix}
    \Big] \Big\rVert_{\widetilde{\mathcalboondox{H}}} \eqsim \Big\lVert \widetilde{B}\widetilde{D} \Big[
    \begin{smallmatrix}
      u \\ 0 \\ 0
    \end{smallmatrix}
    \Big] \Big\rVert_{\widetilde{\mathcalboondox{H}}} \eqsim \Big\lVert \widetilde{D} \Big[
    \begin{smallmatrix}
      u \\ 0 \\ 0
    \end{smallmatrix}
    \Big] \Big\rVert_{\widetilde{\mathcalboondox{H}}} \eqsim \lVert \nabla u \rVert_{L^2(\Omega,W)} + \lVert u \rVert_{L^2(\Omega,\mu)}.
  \end{equation*}
  This concludes the proof,
  since $\sqrt{(\widetilde{B}\widetilde{D})^2}$ applied to $[u \; 0 \; 0]^{\transpose}$
  equals $ [ \sqrt{a L} u \; 0 \; 0]^{\transpose}$.
\end{proof}

%

%
%
%
%
%

%
%
%

We end this section with some examples of matrix weights
and discuss when the hypotheses on the manifold $M$ associated with $\mu, W$ are met.
To obtain examples of $\mu, W$, we consider manifolds $M$ embedded in $\mathbb{R}^N$ obtained as graphs of functions
$\varphi \colon \mathbb{R}^d \to \mathbb{R}^m$, with $N = d + m$.
In \cref{thm:main_theorem_higher_dim}
we thus have
\begin{align*}
  \rho \colon  \mathbb{R}^d & \to M  \\
  x & \mapsto ( x, \varphi(x) ) = ( x, y )
\end{align*}
with Jacobian matrix
$\mathrm{d}\rho_x = (I , \mathrm{d}\varphi_x)^\transpose$.
By reverse engineering, we get from $\varphi$ an example of a Riemannian metric on $\mathbb{R}^d$
\begin{equation*}
  g = \mathrm{d}\rho_x^\star \mathrm{d}\rho_x = I + \mathrm{d}\varphi_x^\star \mathrm{d}\varphi_x .
\end{equation*}
For any choice of scalar weight $\mu$,
this yields an example of a matrix weight $W = \mu g^{-1}$.

\begin{example}\label{ex:ellipses}
  Consider the graph of
  \begin{equation}\label{eq:first_graph}
    \varphi(x_1,x_2) = \Big( \frac{x_1}{x_1^2 + x_2^2} , \frac{x_2}{x_1^2 + x_2^2} \Big) = (y_1 , y_2) ,
  \end{equation}
  for $x = (x_1 , x_2) \in \mathbb{R}^2 \setminus \{ (0,0) \} $.
  Here $\rho(x_1, x_2) = (x_1, x_2, \varphi(x_1, x_2) ) $
  and $M \subseteq \mathbb{R}^4$ is complete
  and asymptotically  isometric to $\mathbb{R}^2$ both when $\lvert x \rvert^2 = x_1^2 + x_2^2 \to + \infty$
  and when $\lvert x \rvert^2 \to 0$. %
  Therefore Ricci curvature and injectivity radius is bounded from below by a compactness argument.
  In this case
  \begin{equation*}
    g_{\varphi} = I + \mathrm{d}\varphi_x^\star \mathrm{d}\varphi_x = \left(1 + \frac{1}{\lvert x \rvert^4} \right)
    \begin{bmatrix}
      1 & 0 \\ 0 & 1
    \end{bmatrix}
  \end{equation*}
  is a conformal metric.
  Therefore this only gives scalar weighted examples of $W$ to which \cref{thm:main_theorem_higher_dim} applies.
  To see a more general matrix weight $W$ appear,
  we can tweak \eqref{eq:first_graph} by composing $\varphi$ with a non-conformal diffeomorphism.
  Consider
  \begin{equation*}
    \phi(x_1,x_2) = \Big( h\Big( \frac{x_1}{x_1^2 + x_2^2} \Big) , \frac{x_2}{x_1^2 + x_2^2} \Big)
  \end{equation*}
  where $h(t) = t \sqrt{1+t^2} $, for $t \in \mathbb{R}$.
  Again $M$ is asymptotically isometric to $\mathbb{R}^2$ both as $\lvert x \rvert^2 \to \infty$  and when $\lvert x \rvert^2 \to 0$,
  so the geometric hypotheses on $M$ are satisfied.
  To see that the metric $g_{\phi}$ obtained from $\phi$, and hence the matrix $W$, is not equivalent to a scalar weight,
  we verify that the singular values of $\mathrm{d}\phi_x$ do not have bounded quotient.
  We calculate
  \begin{align*}
    \partial_{x_1} \phi (t,0) & = \Big( h'(1/t) \cdot (-1/t^2) , 0 \Big) \\
    \partial_{x_2} \phi (t,0) & = \Big( 0 , 1/t^2 \Big) 
  \end{align*}
  so the ratio $\lvert \partial_{x_1} \phi \rvert / \lvert \partial_{x_2} \phi  \rvert (t,0) = \lvert h'(1/t) \rvert \eqsim 1/t \to + \infty$ as $t \to 0^+$.
  The geodesic discs in the metric $g_{\phi}$ are shown in \cref{fig:ellipses1}.
\end{example}

To apply \cref{thm:main_theorem_higher_dim} we need that the Riemannian manifold $(M,h)$ satisfies the geometric hypothesis in \cite[Theorem 1.1]{AMR},
namely that the Ricci curvature $\mathrm{Ric}(M)$ is bounded from below and $M$ has a positive injectivity radius.
In a forthcoming paper, we shall however relax the positive injectivity radius assumption in \cite{AMR}, so that \cref{thm:main_theorem_higher_dim} applies to this example.
\begin{example}\label{ex:example2}
  Let $M$ be the graph of the scalar function $\varphi(x,y) = (x^2 + y^2)^{-a}$, for $a >0$.
  The Gaussian and Ricci curvature $\mathrm{Ric}(M)$ goes to $- (x^2 + y^2)^{2a}$ when $x^2 + y^2 \to 0^+$.
  This can be checked via the Brioschi formula for the Gaussian curvature $K$ in terms of the first fundamental form,
  see \cite[page 13]{Gauss1902}.
  For a surface described as graph of the function $z= \varphi(x,y)$, as in our case, we have: 
  \begin{equation*}
    K = \frac{\varphi_{x x} \varphi_{x y}-\varphi_{x y}^2}{\left(1+\varphi_x^2+\varphi_y^2\right)^2}=-\frac{(1+2 a) 4 a^2 (x^2+y^2)^{2 a}}{\left[(x^2+y^2)^{2 a+1}+ 4 a^2\right]^2} \eqsim -\left(x^2+y^2\right)^{2 a}
  \end{equation*}
  as $\lvert(x,y)\rvert \rightarrow 0^{+}$, and $a > 0$.
  So the Ricci curvature is bounded below,
  but the injectivity radius is not bounded away from zero. %
  Indeed, as discussed in \cite[\S 2.1]{AMR},
  the geometric hypothesis in \cite[Theorem 1.1]{AMR} implies in particular
  that geodesic balls of radius $1$ are Lipschitz diffeomorphic %
  to Euclidean balls. But this is not true in this example,
  so \cite{AMR} does not apply to this manifold.
  Geodesic discs in this metric $g_{\phi}$ are shown in \cref{fig:ellipses2}.
\end{example}

\section{Matrix degenerate Boundary Value Problems}
\label{sec:BVPs}

We show in this final section how the methods in this paper yield
solvability estimates of elliptic Boundary Value Problems (BVPs) for matrix-degenerate divergence form equations
\begin{equation}\label{eq:div_form}
  \mathrm{div} A \nabla u = 0
\end{equation}
on a compact manifold $\Omega$ with Lipschitz boundary $\partial \Omega$.
We assume that there exists a matrix weight $V$
that describes the degeneracy of the coefficients $A$,
in the following way.
\begin{lemma}\label{lemma:V-accretivity}%
  Let $V$ be a matrix weight and let $A$ be a multiplication operator.
  The following are equivalent:
  \begin{itemize}
  \item $V^{-1/2}A V^{-1/2}$ is uniformly bounded and accretive;
  \item $V^{-1}A$ is $V$-bounded and $V$-accretive;
  \item for all vectors $v,w \in \mathbb{C}^{d+1}$ we have
    \begin{equation}\label{eq:uniformly_bdd_accretive}
      \Real \langle A v , v \rangle \gtrsim \langle V v , v\rangle\quad \text{ and } \quad \lvert \langle A v , w \rangle\rvert \lesssim \langle V v, w\rangle.
    \end{equation}
  \end{itemize}
\end{lemma}

A weak solution $u$ to \eqref{eq:div_form} is a function such that $\nabla u \in L^2_{\mathrm{loc}}(T\Omega,V)$,
where $T\Omega$ is the tangent bundle on $\Omega$.
Since the weighted space $L^2_{\mathrm{loc}}(T\Omega,V) \hookrightarrow L^1_{\mathrm{loc}}(T\Omega)$,
then $A \nabla u \in L^1_{\mathrm{loc}}(T\Omega)$ and $\nabla u \in L^1_{\mathrm{loc}}(T\Omega)$, so $u \in W^{1,1}_{\mathrm{loc}}(T\Omega)$ by Poincaré inequality.
Further we assume given a closed Riemannian manifold $M_0$ and, for $\delta > 0$, a bi-Lipschitz map
\begin{equation}
  \label{eq:rho_0}
  \begin{aligned}
    \rho_0 \colon [0,\delta) \times M_0 & \to U \subseteq \Omega \, , \\
    (t , x) & \mapsto \rho_0(t,x)
  \end{aligned}
\end{equation}
between a finite part of the cylinder $\mathbb{R} \times M_0$ and a neighbourhood $U$ of the boundary $\partial\Omega$,
so that $\rho_0(\{0\}\times M_0) = \partial\Omega$. See \cref{fig:cylinders}.
When $\partial\Omega$ is a strongly Lipschitz boundary, that is, when $\partial\Omega$ is locally the graph of a Lipschitz function, such map $\rho_0$ can be constructed using a smooth vector field that is transversal to $\partial\Omega$.

To analyse a weak solution $u$ of \eqref{eq:div_form} near $\partial\Omega$,
we define the pullback $u_0 \coloneqq u \circ \rho_0$ on the cylinder $\mathcal{C}_0 \coloneqq [0,\delta) \times M_0$.
Then $u_0$ satisfies
\begin{equation}\label{eq:div_formA_0}
  \mathrm{div}_{\mathcal{C}_0}A_0\nabla_{\mathcal{C}_0} u_0 = 0 ,
\end{equation}
with coefficients
\begin{equation}
  \label{eq:A_0_coefficients}
  A_0 \coloneqq J_{\rho_0} (\rho_0)_*^{-1} A (\rho_0^*)^{-1}
\end{equation}
where $(\rho_0)_*$ denotes the pushforward via $\rho_0$, so $J_{\rho_0}^{-1}(\rho_0)_*(v) \coloneqq J_{\rho_0}^{-1} \mathrm{d}\rho_0 (v \circ \rho_0^{-1})$ is the Piola transformation,
and $\rho_0^* v = (\mathrm{d}\rho_0)^\star v\circ \rho_0 $ denotes the pullback via $\rho_0$.
See \cite[\S 7.2 and Example 7.2.12]{Rosen2019} for more details on this transformation.
The differential operators in \eqref{eq:div_formA_0} are
\begin{equation}\label{eq:full_gradient_and_div}
  \begin{aligned}
    \nabla_{\mathcal{C}_0}u_0 & \coloneqq [\partial_{t} u_0 , \nabla_{M_0}u_0]^{\transpose} , \\
    \mathrm{div}_{\mathcal{C}_0} \vec{v}_0 & \coloneqq \partial_t (e_0 \cdot\vec{v}_0) + \mathrm{div}_{M_0} (\vec{v}_0)_{\Vert} ,
  \end{aligned}
\end{equation}
where $e_0$ denotes the vertical unit vector along the cylinder,
and $(\vec{v}_0)_{\Vert}$ is the tangential part of $\vec{v}_0$.
Define
the pulled-back matrix weight $V_0 \coloneqq J_{\rho_0} (\rho_0)_*^{-1} V (\rho_0^*)^{-1}$.
\begin{lemma}\label{lemma:equivalence_VA_and_V_0A_0}
  The matrix $V^{-1/2} A V^{-1/2}$ is uniformly bounded and accretive
  on a neighbourhood $U$ of the boundary $\partial\Omega$ if and only if
  $V_0^{-1/2}A_0V_0^{-1/2}$ is uniformly bounded and accretive on $[0,\delta) \times M_0$.
\end{lemma}
Indeed,
the condition \eqref{eq:uniformly_bdd_accretive} for $A$ and $V$
is seen to be equivalent to \eqref{eq:uniformly_bdd_accretive} for $A_0$ and $V_0$.
To obtain solvability estimates,
we require that the matrix weight $V_0$ has the structure
\begin{equation}\label{eq:structure_of_V_0}
  V_0(t,x) =
  \begin{bmatrix}
    \mu(x) & 0 \\
    0 & W(x)
  \end{bmatrix},
\end{equation}
meaning that $V_0$ is constant along the cylinder $\mathcal{C}_0$
and that the vertical direction is a principal direction of $V_0$.
The functions $\mu$ and $W$ are assumed to be scalar and matrix weights on $M_0$, respectively.
Using a transformation of coefficients $A \mapsto B$ from \cite{AAMc2010},
the divergence form equation \eqref{eq:div_formA_0} can be turned into an evolution equation
\begin{equation}
  \label{eq:DBevolution}
  (\partial_t + DB)f_0 = 0 
\end{equation}
for the conormal gradient $f_0 \coloneqq [(1/\mu) \partial_{\nu_{A_0}} u_0 , \nabla_{M_0} u_0]^{\transpose}$ of $u_0$ on the cylinder $[0,\delta)\times M_0$.
Here $\partial_{\nu_{A_0}} u_0 \coloneqq e_0 \cdot A_0\nabla_{\mathcal{C}_0}u_0$ is the conormal derivative.
We make this correspondence precise in the following lemma.
\begin{lemma}\label{lemma:equivalence_divform_andCR}
  A function $u_0$ is a weak solution to the divergence form equation
  \begin{equation*}
    \mathrm{div}_{\mathcal{C}_0}A_0\nabla_{\mathcal{C}_0} u_0 = 0 \, , \quad \text{ with } A_0 =
    \begin{bmatrix}
      a & b \\
      c & d
    \end{bmatrix},
  \end{equation*}
  if and only if its conormal gradient $f_0$ solves the Cauchy--Riemann system \eqref{eq:DBevolution}
  with
  \begin{equation*}
    D =
    \begin{bmatrix}
      0 & -(1/\mu) \mathrm{div}_{M_0} W \\
      \nabla_{M_0} & 0
    \end{bmatrix} \, \text{ and }
    B =
    \begin{bmatrix}
      \mu a^{-1} & - a^{-1} b \\
      W^{-1} c a^{-1}\mu & W^{-1} (d - c a^{-1}b) 
    \end{bmatrix}.
  \end{equation*}
  The operator $D$ is self-adjoint on $L^2(M_0,\mu) \oplus L^2(TM_0, W)$ and $B$ is $(\mu\oplus W)$-bounded and $(\mu\oplus W)$-accretive.
\end{lemma}

\begin{proof}
  Consider the transformation of the coefficient $A_0 \mapsto \mathcalboondox{I}(A_0)$ given by
  \begin{equation*}\mathcalboondox{I}\left(
      \begin{bmatrix}
        a & b \\
        c & d
      \end{bmatrix}\right)
    =
    \begin{bmatrix}
      a^{-1} & - a^{-1} b \\
      c a^{-1} & d - c a^{-1}b
    \end{bmatrix}.
  \end{equation*}
  This map is an involution and
  preserves accretivity and boundedness \cite[Proposition 3.2]{AAMc2010}.
  Following \cite{AAMc2010, AMR} the divergence form equation \eqref{eq:div_formA_0} %
  is equivalent to
  \begin{equation}\label{eq:unweighted_DB_system}
    \left( \partial_t +
    \begin{bmatrix}
      0 & -\mathrm{div}_{M_0} \\
      \nabla_{M_0} & 0
    \end{bmatrix}
    \mathcalboondox{I}(A_0)\right)
    \begin{bmatrix}
      \partial_{\nu_{A_0}} u_0 \\ \nabla_{M_0} u_0
    \end{bmatrix} =
    \begin{bmatrix}
      0 \\ 0
    \end{bmatrix}.
  \end{equation}
  Then a computation shows that
   \begin{equation}\label{eq:Involution_with_weights}
     \mathcalboondox{I}\left(
       \begin{bmatrix}
         v_1 & 0 \\
         0 & W_1
       \end{bmatrix}
       \begin{bmatrix}
         a & b \\
         c & d
       \end{bmatrix}
       \begin{bmatrix}
         v_2 & 0 \\
         0 & W_2
       \end{bmatrix}\right) =
     \begin{bmatrix}
       v_2^{-1} & 0 \\
       0 & W_1
     \end{bmatrix}
     \mathcalboondox{I}(A_0)
     \begin{bmatrix}
       v_1^{-1} & 0 \\
       0 & W_2
     \end{bmatrix}.
   \end{equation}
   We introduce weights into the system \eqref{eq:unweighted_DB_system} as following:
   \begin{align*}%
     \begin{bmatrix}
       1/\mu & 0 \\
       0 & I
     \end{bmatrix}
     \Big( \partial_t & + 
    \begin{bmatrix}
      0 & -\mathrm{div}_{M_0} \\
      \nabla_{M_0} & 0
    \end{bmatrix}
    \begin{bmatrix}
      1 & 0 \\
      0 & W
    \end{bmatrix}
    \begin{bmatrix}
      1 & 0 \\
      0 & W^{-1}
    \end{bmatrix}
    \mathcalboondox{I}(A_0)
    \begin{bmatrix}
      \mu & 0 \\
      0 & I
    \end{bmatrix}
    \begin{bmatrix}
      1/\mu & 0 \\
      0 & I
    \end{bmatrix}\Big)
     \\
     & =
    \Big( \partial_t + D
    \begin{bmatrix}
      1 & 0 \\
      0 & W^{-1}
    \end{bmatrix}
    \mathcalboondox{I}(A_0)
    \begin{bmatrix}
      \mu & 0 \\
      0 & I
    \end{bmatrix}\Big)
    \begin{bmatrix}
      1/\mu & 0 \\
      0 & I
    \end{bmatrix},
   \end{align*}
   where 
   we used that multiplication by $(1/\mu)$ and $\partial_t$ commute since $\mu$ is independent of $t$.
   Using \eqref{eq:Involution_with_weights} we define
   \begin{equation*}
     B \coloneqq
     \begin{bmatrix}
       1 & 0 \\
       0 & W^{-1}
     \end{bmatrix}
     \mathcalboondox{I}(A_0)
     \begin{bmatrix}
       \mu & 0 \\
       0 & I
     \end{bmatrix}
     = \mathcalboondox{I}\Big(
     \begin{bmatrix}
       \mu^{-1} & 0 \\
       0 & W^{-1}
     \end{bmatrix}
     A_0
     \begin{bmatrix}
       1 & 0 \\
       0 & I
     \end{bmatrix}\Big).
   \end{equation*}
   The argument of $\mathcalboondox{I}$ on the right-hand side is $(\mu \oplus W)$-bounded and $(\mu \oplus W)$-accretive.
   Since $\mathcalboondox{I}$ preserves accretivity and boundedness,
    $B$ is uniformly bounded and accretive.
   The reader can check that $B$ coincides with the expression in the statement of the lemma.
\end{proof}

We note that $D B$, with $D$ and $B$ from \cref{lemma:equivalence_divform_andCR},
has the same structure as the operators considered in \cref{sec:higher_dim},
if we replace $\mathbb{R}^d$ by a compact manifold $M_0$.
As in \cref{sec:higher_dim}, we use a metric on $M_0$ adapted to the weights $\mu,W$:
we assume the existence of a smooth, closed Riemannian manifold $(M_1,g_1)$
and a $W^{1,1}_{\mathrm{loc}}$-homeomorphism $ \rho \colon M_0 \to M_1 $
such that  the pullback of
the metric $g_1$ on $M_1$ via $\rho$
 is
\begin{equation*}
  g_0 \coloneqq \rho^* g_1 = \mu W^{-1}, %
\end{equation*}
and we defined the scalar weight
\begin{equation}
  \label{eq:define_nu}
  \nu \coloneqq \rho_*\mu/\sqrt{\det g_1}
\end{equation}
on $M_1$,
where $\rho_*\mu = \mu \circ \rho^{-1}$ denotes the pushforward via $\rho$.
We extend the map $\rho$ to a map between the corresponding cylinders by setting
\begin{equation*}
  \begin{aligned}
    \rho_1 \colon [0,\delta) \times M_0 \to & [0,\delta) \times M_1 \\
    (t,x) \mapsto & (t,\rho(x)).
  \end{aligned}
\end{equation*}
The extension of the Riemannian metric on the cylinder
and its pullback via $\rho_1$ are
\begin{equation}\label{eq:extension_of_metric_to_cylinder}
  \widetilde{g_1} \coloneqq
  \begin{bmatrix}
    1 & 0 \\
    0 & g_1
  \end{bmatrix}, \quad
  \widetilde{g_0} = \rho_1^*\widetilde{g_1} \coloneqq
  \begin{bmatrix}
    1 & 0 \\
    0 & \mu W^{-1}
  \end{bmatrix}.
\end{equation}
In the following, the variable $x$ is in $M_0$, while $y = \rho(x) \in M_1$.
We denote by $\mathrm{d}x$, $\mathrm{d}y$ and $\mathrm{d}z$ the Riemannian measures on
$M_0$, $M_1$ and on $\Omega$, respectively. See \cref{fig:cylinders}.
We also denote by $\mathrm{dist}_0$ and $\mathrm{dist}_1$
the distance functions on $M_0$ and $M_1$ induced by $g_0$ and $g_1$.

\begin{figure}[tbh]
    \includesvg[scale=1.0]{figures/cylinders.svg}
    \caption{The neighbourhood $U$ of $\partial\Omega$ in $\Omega$ is transformed by the bi-Lipschitz map $\rho_0^{-1}$ into the cylinder $[0,\delta)\times M_0$ with anisotropic degenerate coefficients $A_0$.
        The coefficients $A_1$ on the cylinder $[0,\delta)\times M_1$ are isotropically degenerate.}\label{fig:cylinders}
\end{figure}

Note that $A_1$ is isotropically degenerate, meaning that $V_1 = \nu I$ is a scalar weight in each component.
Weak solutions to the anisotropically degenerate equation \eqref{eq:div_formA_0}
correspond to weak solutions to an isotropically degenerate equation on $[0,\delta)\times M_1$.
\begin{lemma}
  Define the coefficients $A_1$  on the cylinder $[0,\delta)\times M_1$ by
  \begin{equation*}
    A_1 \coloneqq \frac{1}{J_{\rho_1}}(\rho_1)_* A_0 \rho_1^* = \frac{1}{J_{\rho_1}} \mathrm{d}\rho_1(A_0 \circ \rho_1^{-1})\mathrm{d}\rho_1^\star.
  \end{equation*}
  Then $A_1/\nu$ is uniformly bounded and accretive.
  Moreover, the function $u_1 = u_0 \circ \rho_1^{-1}$ on $\mathcal{C}_1 = (0,\delta)\times M_1$ is a weak solution to
  \begin{equation}\label{eq:isotr_degen_on_N_1}
    \mathrm{div}_{\mathcal{C}_1}A_1\nabla_{\mathcal{C}_1} u_1 = 0
  \end{equation}
  if and only if $u_0$ is a weak solution to
  \begin{equation}\label{eq:div_form_eq_on_N_0}
    \mathrm{div}_{\mathcal{C}_0}A_0\nabla_{\mathcal{C}_0} u_0 = 0
  \end{equation}
  on $\mathcal{C}_0 = (0,\delta)\times M_0$.
\end{lemma}
\begin{proof}
  Define the matrix weight $V_1 \coloneqq \frac{1}{J_{\rho_1}}(\rho_1)_* V_0 (\rho_1)^*$ %
  on $[0,\delta)\times M_1$.
    Replacing $\rho_0^{-1}$ by $\rho_1$ in \cref{lemma:equivalence_VA_and_V_0A_0}
  shows that $V_1^{-1/2} A_1 V_1^{-1/2}$ is uniformly bounded and accretive.
  We have
  \begin{equation*}
    V_1 = \frac{1}{\sqrt{\det g_1}}
    \begin{bmatrix}
      1 & 0 \\
      0 & \mathrm{d}\rho
    \end{bmatrix}
    \begin{bmatrix}
      \mu & 0 \\
      0 & W
    \end{bmatrix}
    \begin{bmatrix}
      1 & 0 \\
      0 & \mathrm{d}\rho^\star
    \end{bmatrix}
    =
    \begin{bmatrix}
      \nu & 0 \\
      0 & \nu I
    \end{bmatrix},
  \end{equation*}
  since $J_\rho = \sqrt{\det g_1}$ and $J_\rho^{-1} \D{\rho} W \D{\rho}^\star = \nu I$. %
  Thus $V_1^{-1/2} A_1 V_1^{-1/2} = A_1/\nu$.
  If $\nabla_{C_0} u_0 \in L^2(V_0)$, then $(\rho_1^{-1})^* \nabla_{C_0} u_0 = \nabla_{C_1}(u_0 \circ \rho_1^{-1})$ is in $L^2(T\mathcal{C}_1, \nu I)$.
  Moreover $A_0 \nabla_{C_0} u_0 \in L^2(T\mathcal{C}_0,V_0^{-1})$,
  so the non-smooth Piola transformation in \cref{prop:chain_rule_vector} shows that
  \begin{equation*}
    \mathrm{div}_{\mathcal{C}_1} \frac{(\rho_1)_*}{J_{\rho_1}} \big(  A_0\nabla_{\mathcal{C}_0} u_0 \big)
    = \frac{(\rho_1)_*}{J_{\rho_1}} \big( \mathrm{div}_{\mathcal{C}_0} A_0 \nabla_{\mathcal{C}_0} u_0 \big) = 0
  \end{equation*}
  in $L^2(\mathcal{C}_1, \nu^{-1})$.
  This completes the proof.
\end{proof}
Since $A_1$ is isotropically degenerate, %
we can apply results from \cite[\S 4]{AMR} %
to obtain solvability estimates of BVPs for $\mathrm{div}_{\mathcal{C}_1}A_1\nabla_{\mathcal{C}_1}u_1 = 0$.
One can then translate to matrix-weighted norms on the cylinder $\mathcal{C}_0$ and in $\Omega$
to obtain the corresponding results for our BVPs for matrix-degenerate equations.
To illustrate this, we consider the $L^2$ non-tangential maximal Neumann solvability estimate
\begin{equation}\label{eq:Neumann_solvability_est}
  \lVert \nabla u \rVert_{\mathcal{X}} \lesssim \lVert \partial_{\nu_{A_0}}u_{\rho}\restriction_M \rVert_{L^2(M,\omega_0^{-1})}
\end{equation}
proved in \cite[Theorem 1.4]{AMR}.
In the notation of the present paper, the right-hand side of \eqref{eq:Neumann_solvability_est} is
\begin{equation*}
 \Big( \int_{M_1} \lvert e_0 \cdot A_1 \nabla_{\mathcal{C}_1} u_1 \rvert^2 \frac{1}{\nu} \D{y} \Big)^{1/2}
\end{equation*}
where $\nabla_{\mathcal{C}_1} u_1$ is the full gradient of $u_1$ as defined in \eqref{eq:full_gradient_and_div}.
Note that
\begin{align}%
  \nabla_{\mathcal{C}_1} u_1 & = (\rho_1^*)^{-1} \nabla_{\mathcal{C}_0} u_0 , \nonumber \\
  \frac1{\nu} \mathrm{d}y & = \Big( \frac{J_{\rho}}{\mu} \Big) (J_{\rho} \mathrm{d}x) = \frac{J_{\rho}^2}{\mu} \mathrm{d}x . \label{eq:1_over_nu_change_variables}
\end{align}
Since $ A_1  = J_{\rho_1}^{-1}(\rho_1)_* \,A_0\, \rho_1^* $,
we get                          %
\begin{equation*}
  e_0 \cdot A_1 \nabla_{\mathcal{C}_1} u_1 = J_{\rho_1}^{-1} (\rho_1^* e_0) \cdot A_0 \nabla_{\mathcal{C}_0} u_0 = J_{\rho_1}^{-1} e_0 \cdot A_0 \nabla_{\mathcal{C}_0} u_0 
\end{equation*}
and since $J_{\rho_1}\restriction_{M_0} = J_{\rho}$, by using \eqref{eq:1_over_nu_change_variables} we have
\begin{equation*}
  \int_{M_1} \lvert e_0 \cdot A_1 \nabla_{\mathcal{C}_1} u_1 \rvert^2 \frac{1}{\nu} \D{y} = \int_{M_0} \lvert e_0 \cdot A_0 \nabla_{\mathcal{C}_0} u_0 \rvert^2 \frac{1}{\mu} \D{x}.
\end{equation*}
As for the left-hand side in \eqref{eq:Neumann_solvability_est},
translating the Banach norm in \cite[eq. (4.13)]{AMR} to our present notation gives
\begin{equation*}
 \lVert \nabla u \rVert_{\mathcal{X}}^2 = \int_{M_1} \lvert \widetilde{N}_*(\eta \nabla_{\mathcal{C}_1} u_1)\rvert^2 \nu \mathrm{d}y + \int_{\Omega}\langle V \nabla u,\nabla u\rangle(1-\eta)^2 \mathrm{d}z ,
\end{equation*}
where $\eta(t)$ is a smooth cut-off towards the top of the cylinder,
for example $\eta(t) = \max\{0, \min(1, 2 - 2t/\delta)\}$.
Note that in the second term, with abuse of notation, we denoted again by $\eta$ the pullback $\eta\circ \rho_0^{-1}$ on $\Omega$.
We recall the definition of the modified non-tangential maximal function $\widetilde{N}_*$ used on the cylinder $[0,\delta)\times M_1$.
\begin{define}[Modified non-tangential maximal function]
  Let $c_0 > 1$, $c_1 >0$ be fixed constants.
  For a point $(t,y) \in [0,\delta)\times M_1$, we define the Whitney region
  \begin{equation*}
    W_1(t,y) \coloneqq (t/c_0 , c_0 t) \times B_1(y,c_1 t)
  \end{equation*}
  where $B_1$ denotes the geodesic ball of $M_1$ with respect to the metric $\mathrm{dist}_1$.
  Then the non-tangential maximal function at a point $y_1 \in M_1$ is
  \begin{equation*}
    \widetilde{N}_* f(y_1) \coloneqq \sup_{t \in (0,c_0 \delta)} \left( \frac{1}{\nu\big(W_1(t,y_1)\big)} \iint_{W_1(t,y_1)} \lvert f(s,y) \rvert^2 \nu(y) \D{s}\D{y} \right)^{1/2},
  \end{equation*}
  where the measure $\nu\big(W_1(t,y_1)\big)$ is taken with respect to the weighted measure $\nu \D{s}\D{y}$
  and equals $t(c_0 -c_0^{-1})\,\nu(B_1)$.
\end{define}
Consider on $M_0$ the distance $\mathrm{dist}_0(x,\xi) \coloneqq \mathrm{dist}_{1}(\rho(x),\rho(\xi))$
which is the geodesic distance on $M_1$ pulled back to $M_0$.
The Whitney regions on $[0,\delta)\times M_0$ %
  are
\begin{equation*}
  W_0(t,x) %
  \coloneqq (t/c_0, c_0 t) \times \{ \xi \in M_0 \,\colon\, \mathrm{dist}_0(x,\xi) < c_1 t \}.
\end{equation*}
Changing variables with $y_1 = \rho(x_1)$,
since $W_0(t,x_1) = \rho_1\big(W_1(t,y_1)\big)$,
we get
\begin{equation}\label{eq:def_mu_measure_of_W_0}
  \iint_{W_1(t,y_1)} \nu(y) \D{s} \D{y} = \iint_{W_0(t,x_1)} \mu(x) \D{s} \D{x} \eqqcolon \mu\big(W_0(t,x_1)\big)
\end{equation}
Changing variables using $\rho_1$
and the expression of the metric $\widetilde{g}_1$ in \eqref{eq:extension_of_metric_to_cylinder},
we also get
\begin{align*}
  \iint_{W_1(t,y_1)} & \lvert \eta(s) \nabla_{\mathcal{C}_1} u_1\rvert^2 \nu(y) \D{s}\D{y} \\
                  & = \iint_{W_0(t,x_1)} \eta(s)^2 \langle \mathrm{d}\rho_1^{-1} (\mathrm{d}\rho_1^{-1})^\star \nabla_{\mathcal{C}_0}(\rho_1^*u_1), \nabla_{\mathcal{C}_0}(\rho_1^*u_1) \rangle \mu(x) \D{s} \D{x} \\
                  &= \iint_{W_0(t,x_1)} \eta(s)^2 \langle
                    \begin{bsmallmatrix}
                      \mu & 0 \\
                      0 & W
                    \end{bsmallmatrix}
                    \nabla_{\mathcal{C}_0} u_0 , \nabla_{\mathcal{C}_0} u_0 \rangle \D{s} \D{x}
\end{align*}
since $\mathrm{d}\rho_1^{-1} (\mathrm{d}\rho_1^{-1})^\star = \widetilde{g}_1^{-1}$ and $\rho_1^*u_1 = u_0$.
We also have
\begin{equation*}
  \int_{M_1} \lvert \widetilde{N}_*(\eta \nabla_{\mathcal{C}_1} u_1)\rvert^2 \nu(y) \D{y} = \int_{M_0} \lvert \widetilde{N}_0(\eta \nabla_{\mathcal{C}_0} u_0)\rvert^2 \mu(x) \D{x},
\end{equation*}
where the new modified non-tangential maximal function is
\begin{equation*}
  \widetilde{N}_0f(x_1) \coloneqq \sup_{t \in (0,c_0\delta)} \left( \frac{1}{\mu\big(W_0(t,x_1)\big)} \iint_{W_0(t,x_1)} \langle 
  \begin{bsmallmatrix}
    \mu & 0 \\
    0 & W
  \end{bsmallmatrix}
  f(s,x),f(s,x)\rangle\D{s}\D{x} \right)^{1/2} ,
\end{equation*}
and $\mu\big(W_0(t,x_1)\big)$ is as in \eqref{eq:def_mu_measure_of_W_0}.
\begin{figure}[tbh]
    \centering
    \includesvg[scale=1.0]{figures/approach_regions.svg}
    \caption{Non-tangential approach regions. On the left the $\mu,W$-adapted approach regions: in the first $\mu W^{-1} \to \infty$ at $M_0$, in the second region $\mu W^{-1} \to 0$.
      On the right-hand side, the corresponding non-tangential conical approach regions to $M_1$.}\label{fig:approach_regions}
\end{figure}

Note that the approach regions for $\widetilde{N}_0$ shown in \cref{fig:approach_regions} left are intimately connected to the failure
of standard off-diagonal estimates for the resolvent of the operator $D B$ from \cref{lemma:equivalence_divform_andCR}.
On the other hand, such off-diagonal estimates do hold for the corresponding operator
associated to $\mathrm{div}_{\mathcal{C}_1}A_1\nabla_{\mathcal{C}_1} u_1 = 0$, from \cite[Proposition 4.2]{AMR}.
And indeed on $M_1$ we have standard non-tangential approach regions on the right in \cref{fig:approach_regions},
and in \cite[Theorem 1.4]{AMR}.

For our solvability result,
we also need the analogue of the \emph{Carleson discrepancy} $\lVert \,\cdot\, \rVert_*$ from \cite[Eq. (4.10)]{AMR}
for a multiplier $\mathcal{E}$ on the cylinder $[0,\delta)\times M_0$ with
Whitney regions $W_0$ and balls $B_0 \subseteq M_0$ taken with respect to the distance $\mathrm{dist}_0(\cdot,\cdot)$.
The quantity $\lVert \mathcal{E} \rVert_*^2$ is given by
\begin{align*}
  \sup_{\substack{\zeta \in M_0\\r <\delta}} \iint_{\left\{\substack{0 < t < r \\ x \in B_0(\zeta,r)}\right\}}\Big(\sup_{(s,\xi)\in W_0(t,x)}\big\lvert V_0(\xi)^{-1/2} \mathcal{E}(s,\xi) V_0(\xi)^{-1/2}\big\rvert\Big)^2 \frac{\D{t}}t \frac{\mu(x) \D{x}}{\mu( B_0(\zeta,r) )} , %
\end{align*}
where $\mu( B_0(\zeta,r) ) = \int_{B_0(\zeta,r)} \mu(x)\D{x}$.

Summarising, we have obtained the following solvability result for the Neumann BVP for anisotropically degenerate divergence form equations \eqref{eq:div_formA_0}.
\begin{theorem}
  Let $\Omega$ be a compact manifold with Lipschitz boundary $\partial\Omega$,
  and let $A$ be a matrix-valued function on $\Omega$ %
  whose degeneracy are described by a matrix-weight $V$, as in one of the conditions of \cref{lemma:V-accretivity}.
  Let $\rho_0$ be the bi-Lipschitz map defined in \eqref{eq:rho_0}.
  Let $A_0 = A_0(t,x)$ and $V_0(t,x)$ be the matrices transformed via $\rho_0$ as in \eqref{eq:A_0_coefficients}
  and assume that $V_0 = \big[
  \begin{smallmatrix}
    \mu & 0 \\
    0 & W
  \end{smallmatrix}\big]
  $
  for a scalar weight $\mu$ and a matrix weight $W$, as in \eqref{eq:structure_of_V_0}.
  Assume that the matrix $A_0(t,x)$ has trace
    \begin{equation*}
      \underline{A}_0(x) \coloneqq A_0(0,x) = \lim_{t\to 0}A_0(t,x).
    \end{equation*}
  Assume the existence of a smooth, closed Riemannian manifold $(M_1,g_1)$
  and a $W^{1,1}_{\mathrm{loc}}$-homeomorphism $ \rho \colon M_0 \to M_1 $ between the manifolds at the base of the cylinders,
  as in \cref{fig:cylinders}.
  We assume also that the scalar weight $\nu$ in \eqref{eq:define_nu} %
  is a Muckenhoupt weight in $A_2(M_1)$. %

  Then there exists $\varepsilon > 0$ depending only on: $[\nu]_{A_2(M_1)}$, $\lVert V_0^{-1/2} A_0 V_0^{-1/2}\rVert_{L^\infty}$
  the accretivity constant of $V_0^{-1/2} A_0 V_0^{-1/2}$,
  and the structural geometric constants of $M_1$:
  dimension, injectivity radius and lower bound on the Ricci curvature;
  such that if
    \begin{enumerate}
    \item the Carleson discrepancy $\lVert A_0 - \underline{A}_0\rVert_* < \varepsilon$,
    \item the trace $\underline{A}_0$ is close to its adjoint as operator on $L^2(\mathcal{C}_0,V_0)$, namely
      \begin{equation*}
        \sup_{x\in M_0} \big\lvert V_0(x)^{-1/2} \big(\underline{A}_0^\star(x) - \underline{A}_0(x)\big)V_0(x)^{-1/2} \big\rvert < \varepsilon 
      \end{equation*}
    \end{enumerate}
  then the Neumann solvability estimate
  \begin{equation*}
    \int_{M_0}\lvert \widetilde{N}_0(\eta \nabla_{\mathcal{C}_0} u_0)\rvert^2 \mu \D{x} + \int_{\Omega} \langle V \nabla u,\nabla u\rangle (1-\eta)^2\D{z}
    \lesssim \int_{M_0} \lvert \partial_{\nu_{A_0}} u_0 \rvert^2 \frac{1}{\mu} \D{x}
  \end{equation*}
  holds for all weak solutions $u$ to $\mathrm{div}A\nabla u = 0$ %
  in $\Omega$, with near boundary values $u_0$ of $u$, in $\mathcal{C}_0$, as above.

\end{theorem}
\begin{proof}
  Apply \cite[Theorem 1.4]{AMR} to the isotropically degenerate equation \eqref{eq:isotr_degen_on_N_1}  on $[0,\delta)\times M_1$
  (see \cref{fig:cylinders}). Translation of this result to the anisotropically degenerate equation $\mathrm{div}A\nabla u = 0$
  in $\Omega$ (and the Lipschitz equivalent equation $\mathrm{div}_{\mathcal{C}_0}A_0\nabla_{\mathcal{C}_0}u_0 = 0$ on the cylinder $[0,\delta)\times M_0$, near $\partial\Omega$)
  gives the stated result.
  We have seen above the translation of the solvability estimate.
  The translation of the Carleson discrepancy and almost self-adjointness hypothesis is done similarly
  using \cref{lemma:equivalence_VA_and_V_0A_0} with $A,A_0$ replaced by $A_1, A_0$
  and a change of variables in the integrals. 
\end{proof}

The solvability estimates for the $L^2$ Dirichlet and Dirichlet regularity BVPs from \cite[Theorem 1.4]{AMR}
and the Atiyah--Patodi--Singer BVPs from \cite[Theorems 4.5,4.6]{AMR} can similarly be extended to anisotropically degenerate equations.
We leave the details to the interested reader.

\appendix

\section{\texorpdfstring{$W^{1,1}$}{Sobolev} pullbacks and Piola transformations}

We generalise the commutation theorem \cite[Theorem 7.2.9, Lemma 10.2.4]{RosenGMA}
for external derivatives and pullbacks to $W^{1,1}_{\mathrm{loc}}$ homeomorphisms and weighted $L^2$ fields.
(We only deal with the scalar and vector case which we need, and only on $\mathbb{R}^d$).
Throughout this section,
$\rho \colon \mathbb{R}^d \to \mathbb{R}^d$ is assumed to be a $W^{1,1}_{\mathrm{loc}}$ homeomorphism,
meaning that $\rho$,$\rho^{-1}$ are continuous with the weak Jacobian matrices %
$\mathrm{d}\rho$,$\mathrm{d}\rho^{-1}$ in $L^1_{\mathrm{loc}}$.

\begin{theorem}[Change of variables]\label{thm:Sobolev_change_of_variables}
  If $\rho$ is a $W^{1,1}_{\mathrm{loc}}$ homeomorphism and let $\Omega \subseteq \mathbb{R}^d$ be an open set, then
  \begin{equation*}
    \int_{\Omega} f(\rho(x)) J_\rho(x) \D{x} = \int_{\rho(\Omega)} f(y) \D{y}
  \end{equation*}
  holds for all integrable, compactly supported functions $f$.
\end{theorem}
See \cite[Theorem 2 and \S 3]{zbMATH00720741} for a proof.

For $f \in C^\infty_c(\mathbb{R}^d)$ and $h \in C^{\infty}_c(\mathbb{R}^d;\mathbb{R}^d)$,
the chain rule in the weak sense reads
\begin{equation}\label{eq:chain_scalar}
  - \int f(\rho(x)) \mathrm{div}h(x) \D{x} = \int (\mathrm{d}\rho_x)^\star (\nabla f)(\rho(x)) h(x) \D{x} .
\end{equation}
This holds for $W^{1,1}_{\mathrm{loc}}$ homeomorphism $\rho$,
as readily seen by mollifying $\rho$ and passing to the limit.
We first extend to non-smooth $f$:
\begin{theorem}[Non-smooth chain rule]\label{prop:chain_rule_scalar}
  Assume $v,V \in A_2^{\mathrm{loc}}$ and $f \in L^2(v)$ is compactly supported with weak gradient $\nabla f \in L^2(V)$.
  Let $\rho$ be a $W^{1,1}_{\mathrm{loc}}$ homeomorphism.
  Define the weights
  \begin{equation*}
    v_\rho(x) \coloneqq J_\rho(x) v(\rho(x)), \quad V_\rho(x) \coloneqq J_\rho(x) \mathrm{d}\rho_x^{-1} V(\rho(x)) (\mathrm{d}\rho_x)^\star)^{-1}
  \end{equation*}
  and assume $v_\rho,V_\rho \in A_2^{\mathrm{loc}}$.
  Then $\rho^*f = f \circ \rho \in L^2(v_\rho)$ has weak gradient
  \begin{equation*}
    \nabla(\rho^* f) = \rho^* \nabla f = \mathrm{d}\rho^\star(\nabla f \circ \rho) \in L^2(V_\rho).
  \end{equation*}
\end{theorem}
\begin{proof}
  Mollify $f_t \coloneqq \eta_t * f$,
  so that $\nabla f_t = \eta_t * \nabla f$.
  It follows that $f_t \to f$ in $L^2(v)$ and $\nabla f_t \to \nabla f$ in $L^2(V)$
  using dominated convergence and bounds for
  the vector Hardy--Littlewood maximal operator introduced by Christ and Goldberg \cite{ChristGoldberg01},
  see \cite[Theorem 3.2]{Goldberg03} and \cref{appx:approximation_in_weighted_Sobolev}.
  Note that $\lVert \rho^* f\rVert_{L^2(v_\rho)} = \lVert f \rVert_{L^2(v)}$ and
  $\lVert \rho^*(\nabla f)\rVert_{L^2(V_\rho)} = \lVert \nabla f \rVert_{L^2(V)}$.
  Apply the chain rule \eqref{eq:chain_scalar} to $f_t$ and $\rho$ for a fixed test function $h$.
  We can pass to the limit in $t$ and conclude since
  the left-hand side of \eqref{eq:chain_scalar} is bounded as
  \begin{align*}
    \int \lvert f_t(\rho(x)) - f(\rho(x)) \rvert v_\rho(x) \D{x} & \lesssim \Big( \int \lvert f_t(\rho(x)) - f(\rho(x)) \rvert^2 v_\rho(x) \D{x} \Big)^{1/2}  \\
    & = \Big( \int \lvert f_t(y) - f(y) \rvert^2 v(y) \D{y}\Big)^{1/2} \to 0
  \end{align*}
  where the first integral is on the compact support of $h$
  and we used \cref{thm:Sobolev_change_of_variables} when changing variables.
  For the right-hand side in \eqref{eq:chain_scalar},
  using that $\lvert V_\rho^{-1} \rvert \in L^1_{\mathrm{loc}}$, we bound
  \begin{align*}
    \int \lvert \langle V_\rho^{1/2} \big( \rho^*(\nabla f_t) - \rho^*(\nabla f) \big) , V_\rho^{-1/2} h \rangle\rvert \D{x} & \lesssim \Big( \int \lvert V_\rho^{1/2} \big( \rho^*(\nabla f_t) - \rho^*(\nabla f) \big) \rvert^2 \D{x} \Big)^{1/2} \\
    & = \Big( \int \lvert V^{1/2} ( \nabla f_t - \nabla f ) \rvert^2 \D{y} \Big)^{1/2} \to 0 .
  \end{align*}
  This concludes the proof.
\end{proof}

Changing variables in \eqref{eq:chain_scalar} gives
\begin{equation}\label{eq:dual_chain_rule}
  - \int f(y) \frac{1}{J_\rho(\rho^{-1}(y))} ( \mathrm{div} h )(\rho^{-1}(y)) \D{y} = \int \nabla f(y) \cdot \Big( \frac{1}{J_\rho} \mathrm{d}\rho \, h \Big)(\rho^{-1}(y)) \D{y}.
\end{equation}
We refer to the transformation applied to $h$ on the right-hand side of \eqref{eq:dual_chain_rule}
as the Piola transformation $J_\rho^{-1} \rho_*$,
where $\rho_*$ denotes the pushforward via $\rho$.
This transformation is the adjoint of the pullback $\rho^*$ with respect to the unweighted $L^2$ pairing.

We extend identity \eqref{eq:dual_chain_rule} to non-smooth vector fields $h$.
\begin{theorem}[Non-smooth Piola transformation]\label{prop:chain_rule_vector}
  Assume that $v,V \in A_2^{\mathrm{loc}}$ and $h \in L^2(\mathbb{R}^d;\mathbb{R}^d, V)$ is compactly supported with
  weak divergence $\mathrm{div}h \in L^2(\mathbb{R}^d,v)$.
  Let $\rho$ be a $W^{1,1}_{\mathrm{loc}}$ homeomorphism.
  Define the weights
  \begin{equation*}
    v^\rho(y) \coloneqq J_\rho(\rho^{-1}(y)) v(\rho^{-1}(y)), \quad V^\rho(y) \coloneqq \big( J_\rho (\mathrm{d}\rho^\star)^{-1} V \mathrm{d}\rho^{-1} \big) \circ \rho^{-1}(y)
  \end{equation*}
  and assume $v^\rho,V^\rho \in A_2^{\mathrm{loc}}$.
  Then $J_\rho^{-1} \rho_* h = \big(\frac{1}{J_\rho} \mathrm{d}\rho h \big) \circ \rho^{-1} \in L^2(V^\rho)$
  and has weak divergence
  \begin{equation*}
    \mathrm{div}(J_\rho^{-1} \rho_* h) = \Big(\frac{1}{J_\rho} \mathrm{div} h \Big) \circ \rho^{-1} \in L^2(v^\rho).
  \end{equation*}
\end{theorem}
\begin{proof}
  The proof is analogous to the one of \cref{prop:chain_rule_scalar}.
  We mollify $h_t \coloneqq \eta_t * h$ component-wise 
  so that $h_t \to h$ in $L^2(\mathbb{R}^d;\mathbb{R}^d, V)$ and $\mathrm{div}h_t \to \mathrm{div}h$ in $L^2(\mathbb{R}^d,v)$
  by dominated convergence and bounds for the vector Hardy--Littlewood maximal operator,
  as in the proof of \cref{thm:approx_in_weighted_Sobolev}.
  Apply the chain rule \eqref{eq:dual_chain_rule} to $h_t$ and $\rho$ for a fixed test function $f$.
  We pass to the limit in $t$ and note that,
  since $(v^{\rho})^{-1} \in L^1_{\mathrm{loc}}$,
  the left-hand side of \eqref{eq:dual_chain_rule} is bounded by
  \begin{align*}
    \Big( \int \Big\lvert \frac{1}{J_\rho(\rho^{-1}(y))} & ( \mathrm{div}h_t - \mathrm{div}h )(\rho^{-1}(y)) \Big\rvert^2 v^{\rho}(y) \D{y} \Big)^{1/2} \\
    & = \Big( \int \lvert (\mathrm{div}h_t - \mathrm{div}h)(x) \rvert^2 v(x) \D{x} \Big)^{1/2} \to 0
  \end{align*}
  where the first integral is on the compact support of the test function $f$
  and then used \cref{thm:Sobolev_change_of_variables} 
  and $\lVert J_{\rho}^{-1} \rho_* \big(\mathrm{div}h\big) \rVert_{L^2(v^{\rho})} = \lVert \mathrm{div}h \rVert_{L^2(v)} $ when changing variables.
  For the right-hand side of \eqref{eq:dual_chain_rule},
  since $\lvert (V^{\rho})^{-1} \rvert \in L^1_{\mathrm{loc}}$,
  we bound
  \begin{align*}
    \int & \Big\lvert (V^{\rho})^{-1/2} \nabla f \cdot (V^{\rho})^{1/2}\Big( \frac{\rho_*}{J_\rho}(h_t) - \frac{\rho_*}{J_\rho}(h) \Big) \Big\rvert \D{y} \\
    & \lesssim \Big(\int \Big\lvert (V^{\rho})^{1/2}\Big(\frac{\rho_*}{J_\rho}(h_t - h)\Big) \Big\rvert^2 \D{y} \Big)^{1/2}
    =  \Big(\int \big\lvert V^{1/2}\big(h_t - h\big) \big\rvert^2 \D{x} \Big)^{1/2} \to 0
  \end{align*}
  where we used that $\lVert J_\rho^{-1} \rho_* h\rVert_{L^2(V^{\rho})} = \lVert h \rVert_{L^2(V)}$.
  This concludes the proof.
\end{proof}

\label{appx:chainrule}

\section{Approximation in weighted Sobolev spaces}

We include a generalisation to our two weights and matrix weight setting
of the classical mollification argument due to Friedrichs. %
Given an open set $\Omega \subseteq \mathbb{R}^d$, let $\mu$ and $W$ be a scalar and matrix weights respectively, both in $A_2^{\mathrm{loc}}(\Omega)$ as in \cref{def:local_A_2}.
Consider the weighted Sobolev space
\begin{equation*}
  \mathcalboondox{H}(\Omega) \coloneqq H^1_{(\mu,W)}(\Omega) \coloneqq \{ u \in L^2(\Omega,\mu) : \text{ such that } \nabla u \in L^2(\Omega;\mathbb{R}^d,W) \}. %
\end{equation*}
The space $\mathcalboondox{H}_{\mathrm{loc}}(\Omega)$ is defined analogously by requiring that
$u$ and the weak gradient $\nabla u$ are in the corresponding spaces $L^2_{\mathrm{loc}}(\Omega,\mu)$ and $L^2_{\mathrm{loc}}(\Omega;\mathbb{R}^d,W)$.

We consider a local version of
the vector Hardy--Littlewood maximal operator $M_W$ introduced by Christ and Goldberg \cite{ChristGoldberg01}:
for $\Omega \subseteq \mathbb{R}^d$ let %
\begin{equation}\label{eq:vector_H-L}
  M_W^{\Omega}(\vec{u})(x) \coloneqq \sup_{\substack{B \ni x \\ B \subset \Omega}}\fint_B \lvert W^{1/2}(x) W^{-1/2}(y) \vec{u}(y) \rvert \D{y} 
\end{equation}
where the supremum is over balls containing $x$, which are contained in $\Omega$.
The operator $M_W$ is bounded from $L^2(\mathbb{R}^d;\mathbb{R}^d)$ to $L^2(\mathbb{R}^d)$ \cite[Theorem 3.2]{Goldberg03},
equivalently,
if $M$ is the Hardy--Littlewood maximal operator, then
$M(\lvert W^{1/2}(x) \,\cdot\,\rvert)$ maps vector-valued functions in $L^2(\mathbb{R}^d;\mathbb{R}^d,W)$ to scalar functions in $L^2(\mathbb{R}^d)$.
In particular, we will use that $M^{\Omega}(\lvert W^{1/2}(x) \,\cdot\,\rvert) \colon L^2(\Omega;\mathbb{R}^d,W) \to L^2(\Omega)$.

\begin{theorem}[Muckenhoupt, Friedrichs]\label{thm:approx_in_weighted_Sobolev}
  Let $\Omega \subseteq \mathbb{R}^d$, and let $\mu$ and $W$ be scalar and matrix weights in $A_2^{\mathrm{loc}}(\Omega)$.
  Then for any $u \in \mathcalboondox{H}_{\mathrm{loc}}(\Omega)$
  there exists a sequence $\{u_n\}_{n \in \mathbb{N}}$ in $C^\infty_c(\mathbb{R}^d)$
  such that
  \begin{align*}                
    u_n \to u \qquad & \text{ in } L^2_{\mathrm{loc}}(\Omega,\mu) \\
    \nabla u_n \restriction_{\omega} \to \nabla u \restriction_{\omega} \quad & \text{ in } L^2(\omega;\mathbb{R}^d,W) \text{ for all } \omega \subset \Omega
  \end{align*}
  where $\omega$ has compact closure inside $\Omega$, and $u \restriction_{\omega}$ is the restriction of $u$ to the set $\omega$.
\end{theorem}

\begin{proof}
  We create the approximating sequence $u_n$ by mollification.
  Let $\mathcalboondox{f}$ be the approximation of the identity
  \begin{equation*}
    \mathcalboondox{f}(x) \coloneqq
    \begin{cases}
      c e^{-1/(1-\lvert x \rvert^2)} & \text{ for } \lvert x \rvert < 1 \\
      0 & \text{ otherwise }
    \end{cases}
  \end{equation*}
  where the constant $c$ is chosen so that $\mathcalboondox{f}$ is normalised in $L^1$.
  Then $\mathcalboondox{f} \in C^\infty_c(B(0,1))$ is radially decreasing on $\mathbb{R}^d$.
  Let $\mathcalboondox{f}_t$ be the $L^1$-rescaling $\mathcalboondox{f}(x/t) 1/t^d$.
  By \cite[Corollary 2.1.12]{GrafakosClassical}  for any locally integrable function $u$ it holds that
  \begin{equation}\label{eq:bound_by_HL}
    \sup_{t > 0} (\mathcalboondox{f}_t * \lvert u \rvert)(x) \le M u(x) \quad \text{ for a.e. } x \in \mathbb{R}^d
  \end{equation}
  where $M$ is the Hardy--Littlewood maximal operator.
  Since $\mathcalboondox{f}$ is an approximation of the identity,
  the sequence $u_t \coloneqq u * \mathcalboondox{f}_t$ converges pointwise almost everywhere to $u$ as $t\to 0$.
  The bound \eqref{eq:bound_by_HL} provides a domination in $L^2(\omega,\mu)$ for any compact subset $\omega \subset \Omega$ and for any weight $\mu \in A_2^{\mathrm{loc}}(\Omega)$,
  so we can conclude via the dominated convergence theorem that
  $u_n \to u$ in  $L^2_{\mathrm{loc}}(\Omega,\mu)$.

  For the convergence of $\nabla u_t$ to $\nabla u$ in $L^2(\Omega;\mathbb{R}^d,W)$,
  we extend the bound in \eqref{eq:bound_by_HL} using the local vector maximal operator in \eqref{eq:vector_H-L}.
  For a vector-valued function $v$, the convolution $\mathcalboondox{f} * v$ is intended component-wise. 
  Notice that by  linearity of the convolution, for any matrix-valued function $A(x)$ we have
  \begin{equation*}
    A(x) (\mathcalboondox{f} * v)(x) = (\mathcalboondox{f} * A(x)v)(x) = \int_{\mathbb{R}^d} \mathcalboondox{f}(x-y) A(x)v(y) \D{y}.
  \end{equation*}
  Moreover, since all norms on a finite dimension vector space are equivalent and $\mathcalboondox{f}$ is non-negative, we have 
  \begin{equation*}
    \lvert (\mathcalboondox{f} * v)(x) \rvert \eqsim \sum_{j=1}^d \lvert (\mathcalboondox{f} * v_j)(x) \rvert \le (\mathcalboondox{f} * \lvert v \rvert)(x).
  \end{equation*}
  Putting these two estimates together, we can apply the bound \eqref{eq:bound_by_HL}
  to obtain
  \begin{align*}
    \lvert \sup_{t>0} A(x) (\mathcalboondox{f}_t * v)(x) \rvert & \le \sup_{t>0} (\mathcalboondox{f}_t * \lvert A(x) v \rvert)(x)  \\
                                                             & \le M(\lvert A(x) v \rvert)(x)
  \end{align*}
  for almost every $x$. The local vector maximal operator $M^{\Omega}(\lvert W^{1/2}(x) \,\cdot\,\rvert)$ is bounded from $L^2(\Omega;\mathbb{R}^d,W)$ to $L^2(\Omega)$.
  We can conclude
  by dominated convergence, which amounts to applying Fatou's lemma to the following non-negative scalar sequence %
  \begin{equation*}
    2^2 \big\lvert M^{\Omega}(\lvert W^{1/2}(x)\nabla u\rvert)(x) \big\rvert^2 - \big\lvert W^{1/2}(x) \big( (\mathcalboondox{f}_n * \nabla u)(x) - \nabla u(x) \big) \big\rvert^2.
  \end{equation*}
  This concludes the proof.
\end{proof}

\label{appx:approximation_in_weighted_Sobolev}

\section*{Acknowledgements}
G.B. was supported by the Knut and Alice Wallenberg foundation,
KAW grant 2020.0262 %
postdoctoral program in Mathematics for researchers from outside Sweden.
A.R. was supported by Grant 2022-03996 from the Swedish research council, VR.
We thank the anonymous referees for comments that improved the article.

\printbibliography%

\end{document}